\documentclass[tbtags]{amsart}

\usepackage{amssymb, enumitem, cancel}
\usepackage{geometry}

\usepackage{mathtools}
\usepackage{mathrsfs}
\usepackage{hyperref}
\usepackage{cleveref}
\usepackage{xfrac}
\usepackage{marginnote}
\usepackage[foot]{amsaddr}

\newtheorem{thm}{Theorem}[section]
\newtheorem{prop}[thm]{Proposition}
\newtheorem{lem}[thm]{Lemma}
\newtheorem{cor}[thm]{Corollary}
\crefname{thm}{Theorem}{Theorems}
\crefname{prop}{Proposition}{Propositions}
\crefname{lem}{Lemma}{Lemmas}
\crefname{cor}{Corollary}{Corollaries}

\crefname{claim}{Claim}{Claims}
\crefname{step}{Step}{Steps}

\theoremstyle{definition}
\newtheorem{defn}[thm]{Definition}
\newtheorem{ex}[thm]{Example}
\crefname{defn}{Definition}{Definitions}
\crefname{ex}{Example}{Examples}

\theoremstyle{remark}
\newtheorem{rmk}[thm]{Remark}
\crefname{rmk}{Remark}{Remarks}

\newtheorem*{theorem*}{Theorem}

\newtheoremstyle{named}%
    {}{}{}{}{\bfseries}{.}{.5em}{\thmnote{#3}}
\theoremstyle{named}
\newtheorem*{namedthm}{Theorem}

\numberwithin{equation}{section}

\newcommand{\pair}[2]{\ensuremath\langle#1,#2\rangle}
\newcommand{\norm}[1]{\ensuremath\lVert#1\rVert}

\newcommand{\defeq}{\vcentcolon=}
\newcommand{\trinorm}[1]{\ensuremath|\!|\!|#1|\!|\!|}
\newcommand{\eqdef}{=\vcentcolon}

\renewcommand{\leq}{\leqslant}
\renewcommand{\geq}{\geqslant}

\let\temp\phi
\let\phi\varphi
\let\varphi\temp
\let\temp\varepsilon
\let\varepsilon\epsilon
\let\epsilon\temp

\DeclareMathOperator{\tr}{tr}

\newcommand{\C}{\mathbb{C}}

\newcommand{\N}{\mathbb{N}}
\newcommand{\Pc}{\mathcal{P}}
\newcommand{\Z}{\mathbb{Z}}
\newcommand{\de}{\partial}
\newcommand{\di}{\mathrm{d}}
\newcommand{\R}{{\mathbb{R}}}

\newcommand{\call}[1]{\ensuremath\mathcal{#1}}
\newcommand{\frk}[1]{\ensuremath\mathfrak{#1}}
\newcommand{\scr}[1]{\ensuremath\mathscr{#1}}
\newcommand{\bb}[1]{\ensuremath\mathbb{#1}}
\newcommand{\barr}[1]{\ensuremath\bar{#1}}
\renewcommand{\bar}[1]{\ensuremath\overline{#1}}

\newcommand{\trn}{\mathsf{T}}
\newcommand\Sym{\scr S}
\newcommand{\tmod}{\mathrm{mod}\,}

\newcommand\mres\llcorner

\begin{document}

\date{\today}

\title{Some remarks on Linear-quadratic closed-loop games with many players}
\author[M.\ Cirant]{Marco Cirant}
\author[D.\ F.\ Redaelli]{Davide Francesco Redaelli}
\address{Dipartimento di Matematica ``T.\ Levi-Civita''\\ Università degli Studi di Padova\\ Via Trieste 63, 35121, Padova, Italy}
\email{cirant@math.unipd.it (Cirant)}
\email{redaelli@math.unipd.it (Redaelli)}


\subjclass[2010]{}

\begin{abstract}  
We identify structural assumptions which provide solvability of the Nash system arising from a linear-quadratic closed-loop game, with stable properties with respect to the number of players. In a setting of interactions governed by a sparse graph, both short-time and long-time existence of a classical solution for the Nash system set in infinitely many dimensions are addressed, as well as convergence to the solution to the respective ergodic problem as the time horizon goes to infinity; in addition, equilibria for the infinite-dimensional game are shown to provide $\epsilon$-Nash closed-loop equilibria for the $N$-player game. In a setting of generalized mean-field type (where the number of interactions is large but not necessarily symmetric), directly from the $N$-player Nash system estimates on the value functions are deduced on an arbitrary large time horizon, which should pave the way for a convergence result as $N$ goes to infinity.
\end{abstract}

\maketitle

\tableofcontents

\section{Introduction}

Consider a stochastic differential game with $N$ players, indexed by $i \in [[N]] \defeq \{ 0,\dots, N-1\}$, where the state $X^i$ of the $i$-th player evolves according to the $\R^d$-valued SDE
\[
\di X^i_t = \alpha^i_t \,\di t + \sqrt{2} \,\di B^i_t .
\]
His/her cost, in the fixed time horizon $[0,T]$, is given by
\begin{equation} \label{dir_costi}
J^i(\alpha) = \frac12 \, \bb E \int_0^T\! \big( \lvert\alpha^i\rvert^2 + \pair{F^i X}{X} \big) + \pair{G^i X_T}{X_T},
\end{equation}
for some $F^i = f^i \otimes I_d$ and $G^i = g^i \otimes I_d$, with $f^i \in C^0([0,T];\Sym(N))$ and $g^i \in \Sym(N)$.  The $B^i$'s are independent $\R^d$-valued Brownian motions, and the $\alpha^i$'s are closed-loop controls in feedback form, that is $\alpha^i = \alpha^i(t,X_t)$.

It is known (see for instance \cite[Section 2.1.4]{cardelar}, or \cite{Friedman}) that the value functions of the players, $u^i = u^i(t,x)$ with $i \in [[N]]$, $t \in [0,T]$ and $x = (x^0,\dots,x^{N-1}) \in (\R^d)^N$ solve the so-called Nash system of Hamilton-Jacobi PDEs
\begin{equation} \label{dir_NS_intro}
\begin{cases}
-\de_t u^i - \Delta u^i + \frac12 \lvert D_i u^i \rvert^2 + \sum_{j\neq i} D_j u^j D_j u^i = \bar F^i \\
u^i(T,\cdot) = \bar G^i
\end{cases}
\end{equation}
where $\bar F^i = \frac12 \pair{F^i\cdot}\cdot$, $\bar G^i = \pair{G^i \cdot}{\cdot}$ and $i \in [[N]]$. The equilibrium feedbacks are then given by $\alpha^i = -D_i u^i$. Since the dynamic of each player is linear, and the costs are quadratic in the state and control variables, it is well-known that the previous system can be recast into a system of ODEs of Riccati type, by making the ansatz that $u^i$ are quadratic functions of the states. Here, we look for conditions on the solvability of such system, and we focus in particular on properties that are \textit{stable} as the number of players goes to infinity, with the aim of addressing the limit problem with infinitely many players, whenever possible. The Laplacian appears in \eqref{dir_NS_intro} by the presence of the independent noises $B^i$, but it will not actually play any relevant role in our analysis. We keep it since the purely deterministic system is not known to be well posed, beyond the linear quadratic setting.

The study of differential games with many players has enjoyed a rapid development in the last two decades, since the introduction of the theory of Mean Field Games (MFGs) independently by Lasry and Lions \cite{LL07} and Huang, Caines and Malham\'e \cite{hcm}. MFG theory provides an effective paradigm for studying dynamic games with many players that are both symmetric (that is, indistinguishable) and negligible. In such framework, one has a limit model, involving the equilibrium between a typical player versus a mass of agents, which is decentralised and symmetric: in fact, it is realised by a feedback of the state of the player only, it is identical for all the players, and it can be computed just by observing the population of players at the initial time. We refer to the book \cite{cardelar} and the lecture notes \cite{CPnotes} for a recent account of the theory of MFGs.

If the MFG assumption is not fullfilled, that is, $F^i$ and $G^i$ do not depend just on the empirical measure of the other players, one enters into the broader framework of network games. There has been in the last few years an increasing interest in the understanding of the large population limits of equilibria among players whose interactions are described by graphs. Roughly speaking, when the number of interactions is ``large'', the MFG description, or a suitable generalisation based on the notion of \textit{graphon} effectively characterizes large population limits: see for instance \cite{Carmona2, BWZ, CainesHuang, LackerSoretLabel} and references therein. See also \cite{DelarueESAIM} for the analysis of a model on Erd\H{o}s-R\'enyi graphs.

On the other hand, when the underlying network structure is sparse, very few results are available in the literature. In the dense regime, one expects all the players to have an individual negligible influence on a given player, because their running costs involve cumulative functions of many variables, while in sparse regimes, \textit{most} of the players should have a small impact just because they are ``far'' with respect to the graph distance. This mechanism of independency between players that are far in the graph has been first observed in \cite{LackerSoret}, by means of correlation estimates.

The framework addressed in \cite{LackerSoret} is somehow similar to ours. The linear-quadratic setting is considered, and under some symmetry assumptions on the underlying graph, Nash equilibria are computed explicitly (exploiting also the running costs $\overline F^i$ to be identically zero, and $\overline G^i$ to have a specific structure). Then, a probabilistic information on the covariance of any two players' equilibrium state processes is derived. A main goal in our work is to derive an analogous information in more analytic terms; that is, we wish to quantify the influence of the $j$-th player on the $i$-th one by estimating $D_j u^i$, with the perspective of developing some ideas that could be applied also beyond the linear-quadratic framework. 

\smallskip

We now describe more in details our results. The first part of the paper is devoted to analyse a sort of sparse regime with a special structure, namely \textit{shift-invariance}, where basically $f^i$ and $g^i$ coincide with $f^{i-1}$ and $g^{i-1}$, respectively, after the permutation of variables $x^i \mapsto x^{i-1}$. Most importantly, we assume players to have \textit{nearsighted interactions}, that means
\[
|f^i_{hk}|, |g^i_{hk}| \lesssim \beta_{h-i}\beta_{k-i} \qquad\forall\, {h,k},
\]
where $(\beta_k)_k$ is a suitable sequence in $\ell^1(\Z)$. Since $\beta_{k-i}$ decays as $|k-i|\to \infty$, this means that $f^i_{hk}$, which quantifies the influence of the $h$-th and $k$-th player on the $i$-th one, decreases as $|h-i|$ and $|k-i|$ increase. A prototypical example is given by the bi-directed chain, where $f^{i}_{hk}$ is a sparse matrix with zero entries except for $|h-i| \le 1$ and $|k-i| \le 1$, so that the $i$-th player cost depends only on the $(i+1)$-th and the $(i-1)$-th state. See also \cite{Feng} for results on models that build upon a similar structure.

The first statement we get, which resumes Theorems \ref{thm:locex} and \ref{genlocex}, is the following.

\begin{theorem*} There exists $T^* > 0$ such that if $T \leq T^*$ then for any $N \in \N \cup \{\infty\}$ there exists a smooth solution to system~\eqref{dir_NS_intro} such that, for any $i,j$ and $m \in \N$,
\begin{equation*} 
\bigg\Vert \Big(\frac{\di}{\di t}\Big)^m D_{hk} u^i  \bigg\Vert_{\infty} \lesssim \beta_{h-i} \beta_{k-i}.
\end{equation*}
\end{theorem*}

Note in particular that we get existence of a smooth classical solution to the infinite-dimensional Nash system \eqref{dir_NS_intro} with $N=\infty$, $i \in \Z$. We stress that the key issue in the analysis of our problem is that, despite the cost functions $f^i, g^i$ may depend on very few variables, the system itself is strongly coupled by the transport terms $\sum_{j\neq i} D_j u^j D_j u^i$, which become in fact series when $N = \infty$. The closed-loop structure of equilibria forces the equilibrium feedbacks $\alpha^j = -D_j u^j$ to be strongly ``nonlocal'', that is, they depend on the full vector state $x$. Hence, decay estimates on $D_j u^i$ which are stable as $N$ increases are crucial to pass to the limit $N \to \infty$. These are obtained here by a careful choice of $\beta$: below, we will use the terminology \emph{self-controlled for the discrete convolution}, or briefly \emph{c-self-controlled} (where the ``c'' stands for ``convolution'') , that fits well with the structure of cyclic discrete convolution appearing in our problem. From the game perspective, the equilibrium feedbacks $\alpha^j$ turn out to be almost ``local'', in the sense that the influence of ``far'' players is still negligible in the sense explained above. In other words, despite the strong coupling given by the full information structure of closed-loop equilibria, the property of ``unimportance of distance players'' is observed (see Remark \ref{unimport} for further discussions).
\smallskip

While the shift-invariance condition can be actually dropped (see Section \ref{s:beyond}), one of the main restrictions of the previous result is that it guarantees short-time existence only. Note that, even with a finite number of players $N$, Riccati systems may fail to have long time solutions in general; therefore we look for further conditions on $f^i$ and $g^i$ such that existence holds for any time horizon $T$, and independently on $N$. To achieve this goal, we strengthen the previous assumption on nearsighted interactions, that now become of \textit{strongly gathering} type, and require further \textit{directionality} conditions. Section \ref{s:ltwp} contains precise definition of this notion and examples. The main existence result is stated in Theorem \ref{dir_wp}. Here, we exploit the possibility to relate a solution to system~\eqref{dir_NS_intro} to a flow of generating functions, which works well when $N=\infty$ but has no clear adaption to the finite $N$ setting. 

Within the special setting of systems with cost of strongly gathering type and directionality, we are able to push further our analysis and study the long time limit $T \to \infty$, that is, we show that the value functions $u^i$ converge to solutions of the ergodic problem as the time horizon goes to infinity. To complete this program, estimates on solutions of the Riccati system are obtained at the level of the game with $N=\infty$ players, uniformly with respect to $T$. The main result of this part of the work is stated by Theorem \ref{dir_thm:conv}.

We conclude this second part by showing that equilibria of the infinite-dimensional game provide $\epsilon$-Nash closed-loop equilibria of suitable $N$ player games, see Section \ref{s:almostnash}.

\smallskip

In the last part of the work, we come back to the \textit{dense regime}, and work without any symmetry assumption (like shift-invariance). Our goal is again to deduce some estimates on the Riccati system that do not depend on the number of players, under a \textit{mean-field-like} condition
\[
\sup_i \Big( N \sum_{\substack{h,k\\ k \neq i}} | f^i_{hk} |^2 + N \sum_{\substack{k, k \neq i}} |f^k_{ki} |^2 +  |f^i_{ii} |^2 \Big) \lesssim 1,
\]
and the same for $g^i$. It is crucial to observe that in this dense setting there is no hope to have a limit system of the same form. Indeed, in the sparse case we had an $\ell^1$ control of coefficients independent of $N$ (given by $\beta$), so that the series $\sum_{j\neq i} D_j u^j D_j u^i$ carries to the limit by dominated convergence, while now a dominated convergence cannot performed. In fact, at least in the symmetric case, where the costs are of the form $F^i(x) = V(x^i, \frac1{N-1}\sum_{j\neq i}\delta_{x^j}$), the correct limiting object is the so-called Master Equation, which has been the object of the seminal work \cite{CDLL}. The convergence of $u^i$ to a limit function defined on the space of probability measures has been shown under monotonicity assumptions of $f^i$, and its regularity is obtained as a consequence of the stability properties of the MFG system, which characterizes the limit model.

Here, we wish to deduce the estimates on $u^i$ that allow for a passage to the limit $N\to \infty$ directly on the Riccati system (that is, on the Nash system). For short time horizon, one can basically reproduce the same approach of the previous section (cf.~\Cref{thm:locex_bsiN}). To achieve stability for arbitrary values of $T$, we impose a bound from below on the matrices $(f^i_{ij})_{ij}$ and $(g^i_{ij})_{ij}$:
\begin{equation}\label{fgbelow}
(f^i_{ij})_{ij},\ (g^i_{ij})_{ij} \geq -\kappa I, \quad  \kappa > 0.
\end{equation}
If $\kappa$ is small enough, then the existence of a solution to the Nash system is guaranteed for large $T$, by a mechanism that produces an analogous bound from below on the matrix $(D_{ij} u^i)_{ij}$. Recalling that the equilibrium feedbacks are given by $\alpha^j = - D_j u^j$, this is equivalent to the one-sided Lipschitz condition 
\begin{equation}\label{alphaesti}
 \sum_j \langle \alpha^j(t,x) - \alpha^j (t,y), x^j-y^j \rangle \leq \delta|x-y|^2.
\end{equation}
Our main estimate, which pave the way for a convergence result in the $N\to\infty$ limit, is contained in Theorem \ref{MF_thm}. Interestingly, the monotonicity (or mild non-monotonicity) condition \eqref{fgbelow} generates the nice structural property \eqref{alphaesti} of the equilibrium drift vector  $(-D_ju^j)_j$; 
it is worth observing that, when $\kappa = 0$ and in the symmetric MFG case, \eqref{fgbelow} is equivalent to the \textit{displacement monotonicity condition}  used in \cite{Gangbo} to get well-posedness of the Master Equation.

\smallskip

We now discuss some some further related references. In \cite{Bardi, Feleqi}, the convergence problem of open-loop equilibria in symmetric $N$-person linear-quadratic games is addressed, while \cite{DelarueTch} deals with the selection problem for models without uniqueness. See also \cite{BCCD} for a convergence result in a finite state model. Linear-quadratic MFGs are studied in several works, see for instance \cite{BLY, BSYY, GraberAMO, LMWZ, LSY} and references therein. Further contributions on the convergence problem are contained in \cite{CCP, Djete, Fischer, JacksonTangpi, Lacker}, while recent results on the structure of many players cooperative equilibria can be found in \cite{Cecchin, DDJ, JacksonLacker}.

\smallskip

As a final comment, despite trying to pursue the full generality of the linear-quadratic differential game framework, we carefully analysed here a few case studies, with the aim of highlighting some properties that could be investigated beyond the linear-quadratic setting, and directly on the Nash system of PDEs. With this purpose, we tried to avoid as much as possible the use of explicit formulas. Despite these were quite crucial to study the long time behaviour in the sparse regime, there are a few takeaways here that can be starting points for further investigations. On one hand, we are trying to address the short-time existence of the infinite-dimensional Nash system with non-quadratic cost functions, by estimating the derivatives of $u^i$ by means of the coefficients $\beta$ constructed here. On the other hand, we wish to derive analogous bounds on $u^i$ in the mean-field-like setting, again by working directly on the Nash system, and without making use of limit objects such as the Master Equation.

\medskip

\textit{Acknowledgements.} The authors are partially supported by the Gruppo Nazionale per l'Analisi Matematica, la Probabilit\`a e le loro Applicazioni (GNAMPA) of the Istituto Nazionale di Alta Matematica (INdAM), Italy. Cirant partially supported by the King Abdullah University of Science and Technology (KAUST) project CRG2021-4674 ``Mean-Field Games: models, theory, and computational aspects'', and by the project funded by the EuropeanUnion--NextGenerationEU under the National Recovery and Resilience Plan (NRRP), Mission 4 Component 2 Investment 1.1 - Call PRIN 2022 No.\ 104 of February 2, 2022 of Italian Ministry of University and Research; Project 2022W58BJ5 (subject area: PE - Physical Sciences and Engineering) ``PDEs and optimal control methods in mean field games, population dynamics and multi-agent models''. Redaelli partially supported by Fondazione Cassa di Risparmio di Padova e Rovigo, Italy.

\section{Shift-invariant games}

\subsection{General setting}

For $a \in \Z$ and $b \in \Z_+$, we will use the notation $[a]_b$ to identify the unique natural number $r \in [[b]]$ such that $r = a \ \tmod b$. The entries of a $N \times N$ matrix will be indexed over $[[N]]$. We will denote by $L^{(N)} \in \Sym(N)$ the lower shift matrix $\tmod N$, defined by
\[
L^{(N)}_{hk} = \delta_{h, [k+1]_N} \quad \forall \, h,k \in [[N]],
\]
where $\delta$ is the Kronecker symbol.

Our main assumption in this section, on the structure of the running cost, is that $f^i$ and $g^i$ are shift-invariant, in the sense of the following definition. 
\begin{defn}
A collection of matrices $(M^i)_{i\in[[N]]} \subset \Sym(N)$ is \emph{shift-invariant} if
\begin{equation} \label{dir_a2} \tag{SI}
M^{[i+1]_N} = L^{(N)} M^i {L^{(N)}}{}^\trn \quad \forall \, i \in [[N]].
\end{equation}
\end{defn}
Note that this is equivalent to requiring that for all $i \in [[N]]$ and $x \in \R^N$,
\[
\pair{M^i (x^0, x^1, \ldots, x^{N-1})}{(x^0, x^1, \ldots, x^{N-1})} = \pair{M^{[i+1]_N} (x^1, \ldots, x^{N-1}, x^0)}{(x^1, \ldots, x^{N-1}, x^0)}.
\]

\begin{ex} \label{dir_ex1}
A very basic shift-invariant case is that with $g^0 = 0$ and 
\begin{equation} \label{dir_qtens}
f^0 = w \otimes w,
\end{equation}
with
\[
w = \frac{1}{\ell-1}\,(\ell-1,\underbrace{-1, \dots, -1}_{\ell-1 \ \text{times}}, \underbrace{0, \dots, 0}_{N-\ell \ \text{times}}), \quad \ell \leq N;
\]
that is, by \eqref{dir_a2},
\begin{equation} \label{dir_exdirl1}
\pair{F^iX_t}{X_t} = \bigg\lvert X^i_t - \frac{1}{\ell-1} \sum_{j=[i+1]_N}^{[i+\ell-1]_N} X^j_t \bigg\rvert^2.
\end{equation}

When $\ell = N$,  the cost is actually of Mean-Field type, that is, it penalizes the deviation of the private state $X^i$ from the mean of the vector $X$.

Here $q^0$ induces an underlying directed circulant graph structure $G_\ell$ to the problem; indeed, by assumption \eqref{dir_a2}
\[
A = (f^i_{ij})_{i,j \in [[N]]} - \mathrm{diag}\,(f^i_{ii})_{i\in[[N]]} = (f^i_{ij})_{i,j \in [[N]]} - I_N
\]
can be considered as the asymmetric and circulant adjacency matrix of $G_\ell$, so that \eqref{dir_exdirl1} reads 
\begin{equation} \label{dir_exdirl2}
\pair{F^iX_t}{X_t} = \bigg\lvert X^i_t - \frac{1}{\#\{ j : (i,j) \in G_\ell\}} \sum_{j:\ (i,j) \in G_\ell} X^j_t \bigg\rvert^2.
\end{equation}
The same is true in the more general case when
\begin{equation} \label{dir_wgen}
w = (1,-w_1,\dots,-w_{\ell-1}), \quad \text{with}\ w_j \geq 0 \ \, \forall\, j, \ \ \sum_{j=1}^{\ell-1} w_j = 1;
\end{equation}
here the $w_j$'s are regarded as normalized weights, and we have
\[
\pair{F^iX_t}{X_t} = \bigg\lvert X^i_t - \sum_{j:\ (i,j) \in G_\ell} w_j X^j_t \bigg\rvert^2,
\] 
which generalizes \eqref{dir_exdirl2}.

Note finally that a sort of ``directionality'' is encoded in the above examples, that is, each player $i$ is affected by the ``following'' ones $j > i$ in the chain. This is not yet important at the current stage, namely we may allow for
\begin{equation}\label{wexample}
w = \frac{1}{\ell-1}\,(\ell-1,\underbrace{-1, \dots, -1}_{\ell-1 \ \text{times}}, \underbrace{0, \dots, 0}_{N-\ell-m \ \text{times}}, \underbrace{-1, \dots, -1}_{m \ \text{times}}), \quad \ell+m \leq N;
\end{equation}
It is only from Section \ref{s:ltwp} that $m = 0$ will be required.

\end{ex}

\subsection{The evolutive infinite-dimensional Nash system}

The main object of our study is the Nash system
\begin{equation} \label{dir_NS}
\begin{cases}
-\de_t u^i - \Delta u^i + \frac12 \lvert D_i u^i \rvert^2 + \sum_{j\neq i} D_j u^j D_j u^i = \bar F^i & \text{on $(0,T) \times (\R^d)^N$}\\
u^i(T,\cdot) = \bar G^i
\end{cases}
\end{equation}

When $N = \infty$, we need to be a bit careful about the notion of solution. In this case, $x \in \call X = \ell^\infty(\Z;\R^d)$. Then, we mean that \eqref{dir_NS} admits a classical solution in the following sense.

\begin{defn} \label{dir_def:cs}
A sequence of $\R$-valued functions $(u^i)_{i\in\Z}$ defined on $[0,T] \times \call X$ is a classical solution to the Nash system~\eqref{dir_NS} on $[0,T] \times \call X$ if the following hold:
\begin{enumerate}[label=\textbf{(S\arabic*)}]
\item \label{dir_s1} each $u^i$ is of class $C^1$ with respect to $t \in (0,T)$ and $C^2$ with respect to $x \in \call X$, in the Fréchet sense;
\item \label{dir_s2} for each $i \in \N$, the Laplacian series $\Delta u^i = \sum_j \Delta_j u^i$ and the series $\sum_{j\neq i} D_j u^j D_j u^i$ uniformly converge on all bounded subsets of $[0,T] \times \call X$;
\item \label{dir_s3} system~\eqref{dir_NS} is satisfied pointwise for all $(t,x) \in (0,T) \times \call X$;
\item \label{dir_s4} $u^i(T,\cdot) = \bar G^i$ for all $i \in \N$. 
\end{enumerate}
\end{defn}

As it is customary in linear-quadratic $N$-player games, we look for solutions of the form
\begin{equation*} 
u^i(t,x) = \frac12 \pair{(c^i(t) \otimes I_d)x}x_{\R^{Nd}} + \eta^i(t),
\end{equation*}
for some functions $c^i \colon [0,T] \to \Sym(N)$ such that $c^i(T) = g^i$ and $\eta^i \colon [0,T] \to \R$ which vanish at $T$. We have $D_j u^i(t,x) = {e^j}{}^\trn (c^i(t) \otimes I_d) x$, where $e^j = e_j \otimes I_d$, $\{e_j\}_{j=0}^{N-1}$ being the canonical basis of $\R^N$. Hence, from \eqref{dir_NS} we obtain
\[\begin{multlined}[.95\displaywidth]
x^\trn \bigg( {-\frac12}\, \dot c^i \otimes I_d + \frac12 (c^i \otimes I_d)e^i{e^i}{}^\trn(c^i \otimes I_d) +  \sum_{j\in[[N]]\setminus\{i\}}  (c^j \otimes I_d)e^j{e^j}{}^\trn(c^i \otimes I_d) - \frac12 F^i \bigg) x 
\\ = \tr(c^i \otimes I_d) + \dot \eta^i;
\end{multlined}\]
that is,
\[
x^\trn \bigg(\bigg( {-\frac12 \dot c^i} + \frac12 c^i e_i{e_i}^\trn c^i +  \sum_{j\in[[N]]\setminus\{i\}} c^j e_j{e_j}^\trn c^i - \frac12 f^i \bigg) \otimes I_d \bigg) x = d \tr c^i + \dot \eta^i.
\]
As this must hold for all $x \in \R^{Nd}$, it follows that
\[
\dot \eta^i = - d\sum_{j \in [[N]]} c^i_{jj},
\]
and
\begin{equation} \label{dir_peqc}
{- \dot c^i} + c^i e_i{e_i}^\trn c^i + \sum_{j\in[[N]]\setminus\{i\}} \big( c^j e_j{e_j}^\trn c^i + c^i e_j{e_j}^\trn c^j \big) = f^i,
\end{equation}
as $x^\trn A x = 0$ for all $x$ if and only if $A + A^\trn = 0$ for any square matrix $A$.

Now, given the shift invariance of $f^i$ and $g^i$, one expects a solution to \eqref{dir_peqc} to enjoy the same property, hence we look for a solution such that
\begin{equation} \label{dir_cSI}
c^i = \big({L^{(N)}\big)^i c\, \big({L^{(N)}}{}^\trn}\big)^i,
\end{equation}
for some $c \colon [0,T] \to \Sym(N)$. Clearly, this makes $\eta^i$ independent of $i$, as we will have $\eta^i = \eta \defeq \int_\cdot^T \tr c$. By plugging \eqref{dir_cSI} into \eqref{dir_peqc} and letting $i=0$ one obtains the following system of ODEs for the entries of $c$:
\begin{equation} \label{dir_eqc}
-\dot c_{hk} - c_{0h} c_{0k} + \sum_{j \in [[N]]} \big( c_{0,[h-j]_N} c_{jk} + c_{0,[k-j]_N} c_{jh} \big) = f_{hk}, \quad c_{hk}(T) = g_{hk},
\end{equation}
where $f \defeq f^0$ and $g \defeq g^0$.

As we are interested in the limit problem of infinitely many players, which we expect to be indexed by $\Z$ since we have an undirected structure, it is convenient to shift the indices in such a way that $i=0$ ``stays in the middle''; that is, we let $i=-N',\dots,N''$, instead of $i=0,\dots,N$, where, for example, $N'=N'' = (N-1)/2$ if $N$ is odd, and $N' = N/2 = N''+1$ if $N$ is even. Therefore, we rewrite system~\eqref{dir_eqc} as
\begin{equation} \label{redLQeq0}
-\dot c_{hk} + c_{0h}c_{0k} + \sum_{j\neq0} \left( c_{0,h-j}c_{jk} + c_{hj} c_{0,k-j} \right) = f_{hk}, \quad c_{hk}(T) = g_{hk}, \qquad i=-N',\dots,N'',
\end{equation}
where all indices are understood ${\rm mod} \ N$ and between $-N'$ and $N''$. Now it is immediate to identify a convenient limit system by letting $N \to \infty$.

\subsection{Self-controlled sequences for the discrete convolution}

We can recognise a structure of a cyclic discrete convolution in the sums in~\eqref{dir_eqc}; that is,
\begin{equation} \label{discrcon}
\sum_{j=0}^{N-1} c_{0,[h-j]_N}c_{jk} = (c_{\cdot0} \star_N c_{\cdot k})_h, \qquad \sum_{j=0}^{N-1} c_{hj} c_{0,[k-j]_N} = (c_{h\cdot} \star_N c_{0\cdot})_k.
\end{equation}
We wish to exploit this fact in order to prove existence of the solution to system~\eqref{redLQeq0} for small $T$, for any (possibly infinite) $N$.

Our main tool will be the following.

\begin{defn}
A nonnegative sequence $\beta \in \ell^2(\Z)$ is said to be convolution-self-controlled, or \emph{c-self-controlled}, if $\beta \star \beta \lesssim \beta$; that is,
\begin{equation} \label{eqcsc}
\sum_{j\in\bb Z} \beta_j \beta_{i-j} \leq C \beta_i \quad \forall \, i \in \Z,
\end{equation}
for some constant $C>0$ independent of $i$. 
\end{defn}

Clearly, any nonnegative compactly supported $\beta$ (as a sequence indexed by $\Z$) is c-self-controlled. Nevertheless, we will also be interested in positive c-self-controlled sequences in $\ell^1(\Z)$, which indeed exist, as shown by the following result.

\begin{lem} \label{lem:exseq}
For any $\epsilon > 0$, there exists a positive sequence $\beta \in \ell^1(\bb Z)$, with $\norm{\beta}_2 < \epsilon$ and such that $\beta \star \beta \leq 2\beta$. In particular, one can choose $\beta$ of the form
\begin{equation*} 
\beta_i(\alpha) = \frac{2\alpha}{\alpha^2 + i^2} \left( 1 - (-)^i e^{-\alpha \pi} \right), \quad i \in \bb Z,
\end{equation*}
for some $\alpha = \alpha(\epsilon) > 0$.
\end{lem}

\begin{proof}
It is well-known that the Fourier coefficients of a function $f \in L^2((-\pi,\pi))$, given by
\[
\hat f_j \defeq \frac1{2\pi} \int_{-\pi}^{\pi} f(x) e^{-ijx} \di x, \quad j \in \bb Z,
\]
satisfy the following property: if $f,g \in L^2((-\pi,\pi))$, then $\widehat{fg}_j = (\hat f \star \hat g)_j$ for all $j \in \bb Z$.
Let now be $f_\alpha \defeq e^{-\alpha|\cdot|}$, for any $\alpha > 0$. It is elementary to compute, for each $j \in \bb Z$,
\begin{equation} \label{expexp}
\widehat{f_\alpha}_j =  \frac{2\alpha}{\alpha^2 + j^2} \left( 1 - (-)^j e^{-\alpha \pi} \right) > 0,
\end{equation}
whence $(\widehat{f_\alpha} \star \widehat{f_\alpha})_j = \widehat{f_\alpha^2}_{\!j} = \widehat{f_{2\alpha}}_j \leq 4 \widehat{f_\alpha}_j$,
the last inequality being straightforward to check using the explicit expression~\eqref{expexp}. Also, by Parseval's identity
$\norm{\widehat{f_\alpha}}_{\ell^2(\bb Z)}^2 = \norm{f_\alpha}_{L^2(-\pi,\pi)}^2 = \alpha^{-1}(1-e^{-2\alpha\pi}) \to 0$ as $\alpha \to +\infty$.
so that $\beta = \widehat{f_\alpha}$ has the desired properties for any choice of $\alpha = \alpha(\epsilon)$ sufficiently large.
\end{proof}

\begin{rmk}[Variations on a c-self-controlled sequence]
Clearly any positive multiple of a c-self-controlled sequence stays self-controlled. This allows to have self-controlled sequences of arbitrarily large $\ell^\infty$-norm, although with a larger constant $C$ in \eqref{eqcsc}. On the other hand, one can also build c-self-controlled sequences which decay exponentially faster; indeed, if $\beta$ is c-self-controlled, then for any $\gamma > 0$ so is the sequence defined by setting $\tilde \beta_i \defeq \beta_i e^{-\gamma|i|}$, with the same implied constant $C$. This is easily proven as follows. Suppose that $i\leq0$; then
\begin{equation*} \begin{split}
(\tilde\beta \star \tilde\beta)_i &= e^{\gamma i} \sum_{j > 0}  \beta_j  \beta_{i-j} e^{-2\gamma j} + e^{\gamma i} \sum_{i\leq j \leq 0}  \beta_j  \beta_{i-j} + e^{-\gamma i} \sum_{j < i}  \beta_j  \beta_{i-j} e^{2\gamma j} \\
&\leq e^{\gamma i} \bigg( \sum_{j > 0}  \beta_j  \beta_{i-j} + \sum_{i\leq j \leq 0}  \beta_j  \beta_{i-j} \bigg) + e^{-\gamma i} \sum_{j < i}  \beta_j  \beta_{i-j} e^{2\gamma i} \\
&= e^{\gamma i} \sum_{j\in \bb Z} \beta_j \beta_{i-j} \\
&\leq C\beta_i e^{\gamma i}.
\end{split}\end{equation*}
The case $i\geq0$ is analogous.
\end{rmk}

The next result will be useful to deal with convolution of the form \eqref{discrcon}. It essentially states that c-self-controllability is preserved by suitable perturbations, which include all perturbations with compact support.

\begin{lem} \label{LQ_lembg}
Let $\beta$ be c-self-controlled and let $\theta = (\theta_{hk})_{h,k \in \Z}$ be a nonnegative sequence such that $\theta \lesssim \beta \otimes \beta$.\footnote{We mean $(\beta \otimes \beta)_{hk} = \beta_h \beta_k$ for all $h,k \in \Z$.}
Let $d$ be the sequence given by $d \defeq \beta \otimes \beta + \theta$. Then
\[
(d_{\cdot0} \star d_{\cdot k})_h \lesssim \beta_{h}\beta_{k} \quad \forall \, h,k \in \Z.
\]
\end{lem}

\begin{proof}
It suffices to compute
\[
(d_{\cdot0} \star d_{\cdot k})_h =  \beta_0(\beta \star \beta)_h \beta_k + (\theta_{\cdot0} \star \beta)_h \beta_k + \beta_0 ( \beta \star \theta_{\cdot k})_h + ( \theta_{\cdot0} \star  \theta_{\cdot k} )_h \lesssim \beta_h\beta_k. \qedhere
\]
\end{proof}

\begin{rmk}
We will exploit a straightforward implication of the above inequality, namely that $(d_{\cdot0} \star d_{\cdot k})_h \lesssim d_{hk}$.
\end{rmk}

\subsection{Short-time well-posedness for nearsighted interactions}

We are now ready for our first existence and uniqueness results (\Cref{thm:locex,genlocex} below), which is a direct consequence of the following proposition.

By a game with \emph{nearsighted} interactions we are meaning that $|f| \vee |g| \leq \theta$ pointwise (that is, index-wise) for some $\theta \in \ell^1(\bb Z^2)$ satisfying the hypotheses of \Cref{LQ_lembg}; said differently, we are meaning that $|f| \vee |g| \lesssim \beta \otimes \beta$ for some positive c-self-controlled $\beta \in \ell^1(\Z)$.

\begin{prop} \label{prop:contr}
Let $N \in \N$ and $\beta \in \ell^1(\Z)$ be a positive c-self-controlled sequence. Let $f \in C^0([0,T];\Sym(N))$ and $g \in \Sym(N)$ satisfy $|f| \vee |g| \lesssim \beta \otimes \beta$. Define $\call C_N \defeq C^0([0,T])^{2N+1}$ and write $c = (c_{ij})_{i,j=-N}^{N} \in \call C_N$. For $d = \beta \otimes \beta + |g| \vee \sup_{[0,T]} |f|$, set
\begin{equation} \label{defKN}
\call K_N \defeq \prod_{i,j = - N}^N \big\{ w \in C^0([0,T]) :\ \norm{w}_\infty \leq 2d_{ij} \big\}.
\end{equation} 
Let $J_N \colon \call K_N \to \call C_N$ be the map given, for each $i,j = -N,\dots,N$, by
\begin{equation} \label{defJN}
J_N(c)_{ij}(t) \defeq g_{ij} + \int_t^T \Big( f_{ij} + c_{0i}c_{0j} - ( c_{\cdot0} \star_{2N+1} c_{\cdot j})_i  - ( c_{i\cdot} \star_{2N+1} c_{0\cdot} )_j \Big).
\end{equation}
Then there exist $T^* > 0$, depending on $\beta$ but independent of $N$, such that
\begin{equation*}
T \leq T^*, \quad \implies \quad J_N(\call K_N) \subseteq \call K_N.
\end{equation*}
\end{prop}

\begin{proof}
Let $c \in \call K_N$. If $i \geq 0$,
\begin{equation*} \begin{split}
\left\Vert(c_{\cdot0} \star_N c_{\cdot j})_i\right\Vert_\infty &= \bigg\Vert \sum_{k=-N+i}^N c_{0,i-k}c_{jk} + \sum_{k=-N}^{-N+i-1} c_{0,i-k-2N-1}c_{jk} \bigg\Vert_\infty \\
&\leq 4\sum_{k=-N+i}^N d_{0,i-k}d_{jk} + 4\sum_{k=-N}^{-N+i-1} d_{0,i-k-2N-1}d_{jk} \\
& \leq 4(d_{\cdot0} \star d_{\cdot j})_i + 4(d_{\cdot0} \star d_{\cdot j})_{i-2N-1} \\
&\lesssim d_{ij} + d_{i-2N-1,j},
\end{split} \end{equation*}
where the last estimate comes from~\Cref{LQ_lembg}. As $d \in c_0(\Z^2)$,
\begin{equation*}
\sup_{N \in \N} \, \sup_{0\leq i \leq N} \frac{d_{i-2N-1,j}}{d_{ij}} < \infty,
\end{equation*}
hence $d_{i-2N-1,j} \lesssim d_{ij}$ with an implied constant which does not depend on $N$. The same holds for $-N \leq i < 0$ by a symmetrical argument. Analogously, 
$\left\Vert ( c_{i\cdot} \star_N c_{0\cdot} )_j \right\Vert_\infty \lesssim d_{ij}$ and clearly $\norm{c_{0i}c_{0j}}_\infty \leq 4d_{0i}d_{0j} \leq 4(d_{\cdot0} \star d_{\cdot j})_i \lesssim d_{ij}$.
Therefore,
\begin{equation} \label{samehere}
\norm{J_N(c)_{ij}}_\infty \leq |g_{ij}| + C T d_{ij} \leq (1+C T ) d_{ij} ,
\end{equation}
where the constant $C$ depends only on $\beta$. It follows that for $T>0$ small enough, depending on $\beta$, one has $\norm{J_N(c)_{ij}}_{\infty} < 2d_{ij}$ for all $i,j = -N,\dots,N$.
\end{proof}

\begin{rmk}  \label{rmk:contrinfty}
We stated and proved \Cref{prop:contr} for an odd number of players. This is just a matter of having expressions that look more symmetrical, yet there is no preference about the parity of the number of players, so that the above result holds, \emph{mutatis mutandis}, also if the number of players is even.
It is also clear that, with a very much analogous proof, the thesis of \Cref{prop:contr} also holds for $N = \infty$, where one defines, in a natural way,
\begin{equation*}
\call K_\infty \defeq \prod_{i,j \in \Z} \big\{ w \in C^0([0,T]) :\ \norm{w}_\infty \leq 2d_{ij} \big\}
\end{equation*}
and
\begin{equation*}
J_\infty(c)_{ij}(t) \defeq g_{ij} + \int_t^T \!\big( f_{ij} + c_{0i}c_{0j} - ( c_{\cdot0} \star c_{\cdot j})_i  - ( c_{i\cdot} \star c_{0\cdot} )_j \big).
\end{equation*}
\end{rmk}

\begin{thm} \label{thm:locex}
Under the hypotheses of \Cref{prop:contr}, there exists $T^* > 0$ such that if $T \leq T^*$ then for any $N \in \N \cup \{\infty\}$ there exists a unique smooth solution to system~\eqref{redLQeq0} such that, for any $i,j \in -N',\dots,N''$ and $m \in \N$,
\begin{equation} \label{cijdecr}
\bigg\Vert \Big(\frac{\di}{\di t}\Big)^m c_{ij} \bigg\Vert_{\infty} \lesssim \beta_{i} \beta_{j},
\end{equation}
where the implied constants depend only on $T^*$, $f$, $g$ and $m$.
\end{thm}

\begin{proof}
A fixed point of the map $J_N$ defined in~\eqref{defJN} is a solution. We deal with the case of $J_N$ with $N = \infty$, as the case with $N \in \N$ can be included therein.
Note that $\call K_\infty$ can be considered as a closed ball of the Banach space $\ell^\infty_d(\Z^2; C^0([0,T]))$; that is, the space of functions from $\Z^2$ to $C([0,T])$ with finite norm
\[
\trinorm{\cdot}_\infty \defeq \sup_{i,j \in \Z} d_{ij}^{-1} \norm{\cdot_{ij}}_\infty.
\]
We prove that the map $J_\infty$ is a contraction on $\call K_\infty$, provided that $T$ is sufficiently small with respect to $d$. The conclusion will follow from \Cref{prop:contr}, \Cref{rmk:contrinfty} and the contraction mapping theorem; then, once one have a continuous solution, by the structure of equations~\eqref{redLQeq0}, one bootstraps its regularity up to $C^\infty$, while estimate~\eqref{cijdecr} for $m > 1$ follows by induction differentiating~\eqref{redLQeq0} and estimating as in the proof of \Cref{prop:contr}. Let now $c,\bar c \in \call K_\infty$. We have, for $i,j$ fixed, 
\begin{equation}  \label{JNto0} \begin{split}
\big\Vert J_\infty(\bar c)_{ij} - J_\infty(c)_{ij} \big\Vert_\infty &\leq T \Big( \big\Vert( \bar c_{\cdot0} \star_{N} \bar c_{\cdot j})_i - ( c_{\cdot0} \star_{N} c_{\cdot j})_i \big\Vert_\infty \\ 
&\quad\ + \big\Vert c_{0i}\bar c_{0j} - c_{0i}c_{0j} - \big((\bar c_{i\cdot} \star_{N} \bar c_{0\cdot} )_j - ( c_{i\cdot} \star_{N} c_{0\cdot} )_j \big) \big\Vert_\infty \Big).
\end{split}\end{equation}
We have
\begin{equation*} \begin{split}
\big\Vert( \bar c_{\cdot0} \star \bar c_{\cdot j})_i - ( c_{\cdot0} \star c_{\cdot j})_i \big\Vert_\infty &\leq \sum_{k\in\Z} \Big( \norm{\bar c_{k0}}_\infty\norm{\bar c_{i-k,j} - c_{i-k,j}}_\infty + \norm{\bar c_{k0} - c_{k0}}_\infty \norm{c_{i-k,j}}_\infty \Big) \\
& \lesssim d_{ij} \trinorm{\bar c-c};
\end{split}\end{equation*}
that is, $\trinorm{( \bar c_{\cdot0} \star \bar c_{\cdot j})_i - ( c_{\cdot0} \star c_{\cdot j})_i} \lesssim \trinorm{\bar c-c}$ and analogously for the second term in \eqref{JNto0}. Hence
\[
\trinorm{J_\infty(\bar c) - J_\infty(c)} \lesssim T \trinorm{\bar c - c}. \qedhere
\]
\end{proof}

\begin{thm} \label{genlocex}
Suppose that $f^i$ and $g^i$ are shift-invariant and there exists a positive c-self-controlled $\beta \in \ell^1(\Z)$ such that $|f^0| \vee |g^0| \leq C(\beta \otimes \beta)$, $C > 0$. There exists $T^* > 0$, depending only on $C$ and $\beta$, such that if $T \leq T^*$ then there exists a smooth classical solution to the infinite-dimensional Nash system \eqref{dir_NS} with $i \in \Z$ on $[0,T] \times \call X$.
\end{thm}

\begin{proof}
Let $c$ be the solution given by \Cref{thm:locex}.
For $ x = (x^i)_{i\in\Z} \in \ell^\infty(\Z; \R^d)$, define
\begin{equation*}
U(t,{ x}) = \frac12 \sum_{i,j \in \Z} c_{ij}(t) x^i\! \cdot x^j + \int_t^T \sum_{i\in \Z} c_{ii}(s)\,\di s,
\end{equation*}
where we denoted by $\cdot$ the standard scalar product on $\R^d$.
$U$ is well-defined for $ x \in \ell^\infty(\Z; \R^d)$, and continuous in $t$, because the series normally converge thanks to estimate~\eqref{cijdecr}; for the same reason, also
\begin{equation*}
t \mapsto \de_t^k U(t, x) =  \frac12 \sum_{i,j \in \Z} \Big( \frac{\di}{\di t} \Big)^k c_{ij}(t) x^i\! \cdot x^j - \sum_{i\in \Z} \Big( \frac{\di}{\di t} \Big)^{k-1} c_{ii}(t),  \quad k \in \N,
\end{equation*}
are well-defined and continuous. Finally, for $ h \in \ell^\infty(\Z; \R^d)$, note that (omitting the dependence on $t$)
\begin{equation*}
U( x +  h) - U( x) = \sum_{i,j \in \Z} c_{ij} x^i h^j + \frac12 \sum_{i,j \in \Z} c_{ij} h^i\! \cdot h^j,
\end{equation*}
thus $U(t,\cdot)$ is infinitely many times Fréchet-differentiable in $\ell^\infty(\Z;\R^d)$.
Define now $ u = (u^i)_{i\in\Z}$ by setting
\begin{equation*} 
u^0 \defeq U, \qquad u^{i+1}(t,\underline x) \defeq u^i(t,\sigma \underline x), \quad i \in \Z,
\end{equation*}
where $(\sigma x)_j \defeq x_{j-1}$ for $j \in \Z$.
We have
\begin{equation*}
D_j u^i (t,\underline x) = D_j [ u(t,\sigma^i \underline x) ] = D_{j-i} u(t,\underline x) = \sum_{k \in \Z} c_{j-i,k}(t) x^k, \quad i,j \in \Z,
\end{equation*}
hence $\sum_{j\in\Z} D_ju^j D_j u^i$ locally uniformly converges by estimate~\eqref{cijdecr}. Hence, by construction, $u$ solves \eqref{dir_NS} in the desired sense.
\end{proof}

\begin{rmk}[Unimportance of distance players]\label{unimport}
What essentially allows the Nash system to be well-posed in infinite dimensions is what we call an \emph{unimportance of distant players}, in the sense that the farther a player is from a given one (say the $0$-th player) the smaller the impact is has on it is. This is seen in the fact that, on any bounded $\call B \subset \call X$,
\begin{equation*}
\norm{D_j U}_{\infty, [0,T] \times \call B} = \bigg\Vert \sum_{i\in\Z} c_{ij} x^i \bigg\Vert_{\infty, [0,T] \times \call B} \lesssim \beta_j
\end{equation*}
is infinitesimal as $|j| \to \infty$.
\end{rmk}

\begin{subsection}{Beyond shift-invariance} \label{s:beyond}
We have made the shift-invariance hypothesis to reduce our system of infinitely many equations for $c$ to one equation. Nevertheless, the reader should be aware that the above results can be adapted to a more general setting.

One can suppose that we only have a shift-invariant control on the data; that is
\begin{equation} \label{sii}
|f^i_{hk}| \vee |g^i_{hk}| \lesssim \beta_{h-i}\beta_{k-i}.
\end{equation}
In this case, the most natural limit of \eqref{dir_peqc} is indexed over $\N$ and it suffices to replace $\call K_\infty$ with
\[
\tilde{\call K}_\infty \defeq \big\{ w=(w^i_{hk})_{i,h,k\in\N} \in \ell^\infty(\N^3;C^0([0,T])) :\ \norm{w^i_{hk}}_\infty \leq 2d_{|h-i|,|k-i|} \ \forall \, i,h,k \in \N \big\},
\]
which is a closed ball in $\ell^\infty_{\tilde d}(\N^3;C^0([0,T]))$, letting $\tilde d^i_{hk} \defeq d_{|h-i|,|k-i|}$. Then, for instance, one obtains the following result.

\begin{thm} \label{thm:locex_bsi}
Assume \eqref{sii}. There exists $T^* > 0$ such that if $T \leq T^*$ then for any $N \in \N \cup \{\infty\}$ there exists a unique smooth solution to system~\eqref{dir_peqc} such that, for any $i,h,k \in [[N]]$ and $m \in \N$,
\begin{equation*} 
\bigg\Vert \Big(\frac{\di}{\di t}\Big)^m c^i_{hk} \bigg\Vert_{\infty} \lesssim \beta_{|h-i|}\beta_{|k-i|},
\end{equation*}
where the implied constants depend only on $T^*$, $f$, $g$ and $m$.
\end{thm}

Also, one can fix the dimension $N \geq 2$ and consider $\beta^N \in (\R_+)^{2N+1}$ given by
\[
\beta^N_j \defeq \begin{cases}
1 & \text{if} \ j=0 \\
\dfrac1{N-1} & \text{if} \ |j| \in \{ 1,\dots,N-1 \},
\end{cases}
\]
which is c-self-controlled in the sense that for $|j| \in [[N]]$
\[
(\beta^N \star \beta^N)_j = \sum_{|k| \in [[N]]} \beta^N_k \beta^N_{j-k} = \beta^N_j + \frac1{N-1} \sum_{|k| \in [[N]] \setminus \{0\}} \beta^N_{j-k} \leq \beta^N_j + \frac3{N-1} \leq 4 \beta^N_j.
\]
In this case one can look for a solution to \eqref{dir_peqc} starting with assumption~\eqref{sii} with $\beta = \beta^N$; that is,
\begin{equation} \label{siiN}
|f^i_{hk}| \vee |g^i_{hk}| \lesssim \begin{cases}
1 & \text{if} \ h = i= k\\
N^{-1} & \text{if} \ h \neq i = k \ \text{or vice versa} \\
N^{-2} & \text{if} \ h \neq i \neq k.
\end{cases}
\end{equation}
What we get is the following statement.

\begin{thm} \label{thm:locex_bsiN}
Let $N \in \N$, $N\geq 2$ and assume \eqref{siiN}. There exists $T^* > 0$ such that if $T \leq T^*$ then there exists a unique smooth solution to system~\eqref{dir_peqc} such that, for any $i,h,k \in [[N]]$ and $m \in \N$,
\begin{equation*} 
\bigg\Vert \Big(\frac{\di}{\di t}\Big)^m c^i_{hk} \bigg\Vert_{\infty} \lesssim \begin{cases}
1 & \text{if} \ h = i= k\\
N^{-1} & \text{if} \ h \neq i = k \ \text{or vice versa} \\
N^{-2} & \text{if} \ h \neq i \neq k.
\end{cases}
\end{equation*}
where the implied constants depend only on $T^*$, $f$, $g$ and $m$.
\end{thm}
\end{subsection}
Notice that this can be regarded as a result in a non-symmetric mean-field-like setting (see Remark \ref{mfl} for further comments on this terminology), as assumption~\eqref{siiN} is consistent with the fact that we expect the $j$-th derivative of a mean-field-like cost for the $i$-th player to scale by a factor of $N^{-1}$ whenever $j \neq i$.

\subsection{Long-time well-posedness for shift-invariant directed strongly gathering interactions}\label{s:ltwp}

It is clear that the previous construction follows the standard Cauchy-Lipschitz local (in time) existence argument, and the existence (and uniqueness) of a solution can be as usual continued up to a maximal time $T^*$, that is when the quantity $\max_{i,j} \beta_i^{-1}\beta_j^{-1} |c^0_{ij}|$ blows up. So far, we cannot exclude in general that such blow up time $T^*$ is finite.

Upcoming \Cref{def:sg} introduces an additional assumption on the running and terminal costs which will allow us to prove long-time well-posedness for the infinite-dimensional Nash system.
Given $m \leq n$, we will identify $M \in \Sym(m) \subset \Sym(n)$ by considering $\R^m \simeq \R^m \times \{0\} \subset \R^m \times \R^{n-m} = \R^n$ and then extending $M = 0$ on $\R^{n-m}$. Also, given $M \in \Sym(m)$ and $M' \in \Sym(n)$ we will say that $M=M'$ if $M'$ equals the above-mentioned extension of $M$ over $\R^n$.

\begin{defn} \label{def:sg}
Let $\scr M = (M^{(N)})_{N \in \N}$ be a sequence of matrices, with $M^{(N)} \in \Sym(N)$. We say that $\scr M$ is \emph{directed} if there exists $\ell \in \N$ such that eventually $M^{(N)} \equiv M^{(\ell)} \in \Sym(\ell)$. Given $\varrho > 1$, we say that $M \in \Sym(\ell)$ is $\varrho$-\emph{strongly gathering} if the polynomial
\[
\mu(z,w) = \sum_{h,k=0}^{\ell-1} M_{hk} z^h w^k, \quad z,w \in \C
\]
is such that $\mu(z,0) \notin (-\infty, 0)$ if $z \in \varrho \bar{\bb D}$.\footnote{We will use the notation $\bb D_r(z)$ for the open complex disc of radius $r$ about $z$, but we will omit the center when it is $0$ and the radius when it is $1$, so that, e.g.,~we have $\bb D_r = \bb D_r(0) = r \bb D$.} 
\end{defn}

\begin{rmk} The term ``directed'' is related to the fact that the $i$-th player's cost is affected by the states of the ``following'' players $j>i$ in the chain, and this fact might not be immediately clear from the previous definition, which just requires the matrices $M^{(N)}$ to be extensions of a fixed matrix, not depending on $N$. One may then have a look to the matrices of the form $M = w \otimes w$ in (the end of) \Cref{dir_ex1}. In particular,  $w = w^{(N)}$ in \eqref{wexample} gives rise to a directed family only when $m = 0$.

Moreover, in the situations described in \Cref{dir_ex1}, the associated sequence of matrices as the dimension $N$ diverges is not strongly gathering. Indeed, even though $\ell$ stays bounded, one has that $\mu$ is a polynomial with $\mu(1,0) = 0$. In fact, those situations can be seen as limit settings corresponding to taking $\varrho = 1$; we will comment on this case later on in \Cref{dir_sec:lc}. The validity of the strong gathering assumption as it is, can be achieved in different ways; two basic settings are given in the following examples.
\end{rmk}

\begin{ex}
If we want to stick to a matrix $f^0$ of the form~\eqref{dir_qtens}, in order to have $({f^0}^{(N)})_{N \in \N}$ (where $N$ is the dimension) directed with strongly gathering limit one can require that $\ell$ is independent of $N$ large and 
\[
\sum_{j=1}^{\ell-1} w_j = 1-\epsilon, \quad \epsilon > 0,
\]
so that $\phi(z,0) \geq \epsilon$ if $|z| \leq 1$. Put it differently, it suffices to consider
\[
w = (\nu, -w_1, \dots, -w_{\ell-1}), \quad \text{with} \ w_j \geq 0 \ \forall\, j, \ \sum_{j=1}^{\ell-1} w_j = 1, \ \text{and} \ \nu > 1;
\]
this means that we are considering an underlying graph where each node is directly connected with itself as well, and such link has a negative weight. As in this model a positive weight is associated to the tendency, in order to reduce their cost, of each player to get closer to their neighbors, a negative connection with themself is to be interpreted as a drift towards self-annihilation. More prosaically, this means that the state of each player will also tend to the common position $0 \in \R^d$.

Regarding this example, it is also worth pointing out that the common attractive position $0 \in \R^d$ cannot be any other arbitrary point, in the sense that the structure of problem is not invariant under translation of the coordinates. This is due to the fact that we are considering a graph whose nodes have outdegree different from $1$.
\end{ex}

\begin{ex}
Another setting with directedness and strong gathering is that of
\[
q^0 = \nu E^0 + w \otimes w,
\]
where $w$ can be given by~\eqref{dir_wgen} with $\ell$ independent of $N$, $\nu > 0$ and $E^0x = x^0$ for all $x = (x^0, \dots, x^{N-1}) \in (\R^d)^N$. Here
\[
\pair{F^iX_t}{X_t} = \nu \lvert X^i_t \rvert^2 + \bigg\lvert X^i_t - \sum_{j:\ (i,j) \in G_\ell} w_j X^j_t \bigg\rvert^2,
\]
and with a translation of the coordinates we can lead to this situation also costs like
\[
J^i(\alpha) = \frac12 \, \bb E \int_0^T \!\bigg( \lvert\alpha^i\rvert^2 + \nu \lvert X^i_t - y\rvert^2 + \bigg\lvert X^i_t - \sum_{j:\ (i,j) \in G_\ell} w_j X^j_t \bigg\rvert^2 \,\bigg),
\]
for any given $y \in \R^d$. This example, more genuinely than the previous one, shows that the strong gathering assumption entails that we are giving some sort of preference about where the players should aggregate. This will strengthen the attractive structure yielded by a graph satisfying only assumption \eqref{dir_a2} (\Cref{dir_ex1}), thus providing more stability to our game.
\end{ex}

From this section on in this part, the following assumptions will be in force, declined in suitable manners which will be specified.
\begin{namedthm}[\hypertarget{AS}{Assumptions ($\boldsymbol\star$)}]
The matrices $f^i$ and $g^i$ are shift-invariant and $f^0 = f^0(N)$ and $g^0 = g^0(N)$ are directed with limits $f,g \in \Sym(\ell)$ for some $\ell \in \N$. The matrix $f$ is $\varrho$-strongly gathering for some $\varrho > 0$ and $g$ is \emph{compatible on $I$} with $f$ for some interval $I \subset \R$, in the following sense: given
\[
\varphi (z,w) \defeq \sum_{h,k=0}^{\ell-1} f_{hk} z^h w^k, \qquad \Psi(z,w) \defeq \sum_{h,k=0}^{\ell-1} g_{hk} z^h w^k
\]
and setting
\[
\xi \defeq \sqrt{\phi(\cdot,0)},\footnotemark \qquad \psi \defeq \Psi(\cdot,0),
\]
\footnotetext{The symbol $\sqrt{\cdot}$ denotes the principal branch of the square root function.}
we have
\begin{equation} \label{hpcomp}
\inf_{t \in I} |\psi \tanh(t\xi) + \xi| > 0 \quad \text{on} \ \varrho \bb D.
\end{equation}
\end{namedthm}

\begin{rmk}
Note that for any $t \in \R$, one has that $\psi \tanh(t\xi) + \xi$ is bounded on $D$, as $\xi(z)t$ can be a singular point only if $\xi(z)^2 = \phi(z,0) < 0$, which contradicts that $z \in D$. Also note that condition~\eqref{hpcomp} holds with $I = \R$ if, e.g., $g = 0$ or $|\psi| \geq |\xi|$.
\end{rmk}

Considering system \eqref{redLQeq0}, we see that for any $N$ sufficiently large with respect to $\ell$ we have $f_{hk} = 0 = g_{hk}$ if either $h$ or $k$ is negative; hence the limit system as $N \to \infty$ will be given by
\begin{equation} \label{dir_eqcZ}
-\dot c_{hk} - c_{0h} c_{0k} + \sum_{j \in \Z} \big( c_{0,h-j} c_{jk} + c_{0,k-j} c_{jh} \big) = f_{hk}, \quad c_{hk}(T) = g_{hk},
\end{equation}
where $f_{hk}$ and $g_{hk}$ extend to $h,k \in \Z^2$ by letting $f_{hk} = 0 = g_{hk}$ if $(h,k) \notin [[\ell]]^2$.\footnote{Note that the limit operator $q$ is exactly the matrix $f \in \Sym(\ell)$ seen as embedded into the space of symmetric linear operators on $\R^{\Z}$.} Given this system, another reduction is possible. Since $f_{hk} = 0 = g_{hk}$ whenever $h \wedge k < 0$, the assumption that $c_{hk} = 0$ if $h \wedge k < 0$ is not a priori incompatible with the structure of system~\eqref{dir_eqcz}. Therefore we will look for a solution with this additional property; in this way, the coefficients which are not a priori null, $(c_{hk})_{h,k \in \N}$, will define in a natural way a symmetric operator $c(t) \otimes I_d$ on $\call X$ which vanishes on $\ell^\infty(\Z_{<0};\R^d)$, so that it can also be seen as a trivial extension to $\call X$ of a symmetric operator on $\ell^\infty(\N;\R^d)$. This reduces the system of ODEs in~\eqref{dir_eqcZ} to
\begin{equation} \label{dir_eqcz}
-\dot c_{hk} - c_{0h} c_{0k} + \sum_{j = 0}^h c_{0,h-j} c_{jk} + \sum_{j = 0}^k c_{0,k-j} c_{jh} = f_{hk}, \quad c_{hk}(T) = g_{hk},
\end{equation}
which will be the object of our following study, and will eventually provide a solution with the particular form presented in the upcoming definition.

\begin{defn}
A \emph{quadratic shift-invariant directed (QSD) solution} to \eqref{dir_NS} is a classical solution of the form
\begin{equation} \label{dir_uLQl}
u^i(t,x) = \frac12 \sum_{h,k \in \N} c_{hk}(t) \pair{x^{h+i}}{x^{k+i}}_{\R^d} + \int_t^T \tr c(s)\,\di s,
\end{equation}
for some $c \colon [0,T] \to \ell^1(\N^2) \subset \ell^1(\Z^2)$.\footnote{A natural immersion $\imath \colon \ell^1(\N^2) \hookrightarrow \ell^1(\Z^2)$ is given by $\imath(x)^i = x^i$ if $i \in \N$ and $\imath(x)^i = 0$ otherwise.}
\end{defn}

\begin{thm} \label{dir_wp}
Under Assumptions \textnormal{(\hyperlink{AS}{$\star$})} with $[0,T] \subseteq I$, there exists a unique QSD solution to~\eqref{dir_NS} on $[0,T] \times \call X$.
\end{thm}

The following lemmata, in whose statement the hypotheses of \Cref{dir_wp} will be implied, provide the steps of our proof of \Cref{dir_wp}.

\begin{lem} \label{dir_lem:uniquesolc}
There exists a unique sequence $(c_{hk})_{h,k \in \N} \subset C([0,T]) \cap C^\infty((0,T))$, with $c_{hk} = c_{kh}$, which solves the infinite dimensional system~\eqref{dir_eqcz} on $[0,T]$.
\end{lem}

\begin{proof}
We perform the change of variables $t \mapsto T-t$ and prove that there exists a unique solution on $[0,T]$ with initial condition $c(0)=g$ to the forward system
\begin{equation} \label{dir_eqczf}
\dot c_{hk} - c_{0h} c_{0k} + \sum_{j = 0}^h c_{0,h-j} c_{jk} + \sum_{j = 0}^k c_{0,k-j} c_{jh} = f_{hk}
\end{equation}
which smoothly extends to $\R$. Then, the solution to \eqref{dir_eqcz} on $[0,T]$ with $c(T)=0$ will be given by the restriction to $[0,T]$ of $\hat c(T-{\cdot})$, where $\hat c$ is the unique solution to \eqref{dir_eqczf} on $\R$ with $\hat c(0)=g$.
Note that $\hat c_{00}$ is the solution to the Riccati equation $\dot{\hat c}_{00} + \hat c_{00}^2 = f_{00}$, where $f_{00} = \varphi(0,0) > 0$ by strong gathering, hence
\[
\hat c_{00}(t) = \nu \, \frac{\nu \sinh(\nu t) + \frk g \cosh(\nu t)}{\frk g \sinh(\nu t) + \nu \cosh(\nu t)}, \quad \frk g \defeq g_{00}, \ \nu \defeq \sqrt{f_{00}}\,.
\]
All the other $\hat c_{hk}$'s satisfy first-order ODEs with coefficient which are second-order polynomials depending only on $f_{hk}$ and $\hat c_{h'k'}$ with $(h',k') \prec (h,k)$, where $\prec$ denotes the strict Pareto preference.\footnote{That is, $(h',k') \prec (h,k)$ if and only if $h'\leq h$ and $k'\leq k$ with at least one strict inequality.} Therefore, existence and uniqueness of the solution to the infinite system may be proved by induction. Indeed, suppose $(\hat c_{0k'})_{0\leq k'<k} \subset C^\infty(\R)$ are given, and note that
\[
\dot{\hat c}_{0k} + 2\hat c_{00}\hat c_{0k} = f_{0k} - \sum_{j=1}^{k-1} \hat c_{0,k-j}\hat c_{0j};
\]
then $\hat c_{0k}$ is unique and smooth as well. This proves the existence and uniqueness of $\hat c_{0k}$ for all $k \in \N$; then, looking at this argument as the base step for a new induction over $h$, one proves analogously the existence and uniqueness of $\hat c_{hk}$ for all $h \in \N$ and any $k \in \N$.
Finally, $c_{hk} = c_{kh}$ since equations~\eqref{dir_eqc} are invariant with respect to the swap $(h,k) \mapsto (k,h)$.
\end{proof}

The arguments below show that the coefficients $c_{hk}$ can be thought as derivatives of a generating function $\hat\Xi$, which will play a fundamental role in the long-time analysis.

\begin{lem} \label{dir_lem:xi}
The solution to \eqref{dir_eqcz} on $[0,T]$ is given by
\begin{equation} \label{dir_cderxi}
c_{hk}(t) = \frac1{h!\,k!} \, \frac{\de^{h+k}}{\de z^h \de w^k}\bigg|_{(0,0)} \hat\Xi(T-t), \quad \forall\, t \in [0,T]
\end{equation}
for some function $\hat\Xi \colon I \times \varrho \bb D^2 \to \C$, of class $C^\infty$ with respect to $t \in I$ and analytic in $\varrho \bb D^2$, such that $\hat\Xi(\cdot,z,w) = \hat\Xi(\cdot,w,z)$ and $\hat\Xi(0,\cdot,\cdot) = \Psi$.
\end{lem}

\begin{proof}
Suppose that, for all $z,w \in D$ fixed, there exists a solution $\hat \Xi$ on $I$ to
\begin{equation} \label{dir_eqxi}
\de_t \hat \Xi(t,z,w) - \hat \Xi(t,z,0) \hat\Xi(t,0,w) + \big( \hat \Xi(t,z,0) + \hat \Xi(t,0,w) \big) \hat \Xi(t,z,w) = \varphi(z,w),
\end{equation}
with the desired properties of smoothness, invariant with respect the swap $(z,w) \mapsto (w,z)$ and such that $\hat \Xi(0,z,w) = \Psi(z,w)$.
Then by taking the derivatives $\de^h_z\de^k_w|_{(0,0)}$ one recovers equation~\eqref{dir_eqczf}, and thus the coefficients given by \eqref{dir_cderxi} satisfy \eqref{dir_eqcz} on $[0,T]$.\footnote{In other words, we are saying that $\de^h_z \de^k_w \hat \Xi(t,0,0) = h!\, k!\, \hat c_{hk}(t)$, where the coefficients $\hat c_{hk}$ are those in the proof of \Cref{dir_lem:uniquesolc}.}
To see that \eqref{dir_eqxi} admits such a solution, note that $\hat \Xi(t,z,0)$ solves the Riccati equation $\de_t \hat \Xi(t,z,0) + \hat \Xi(t,z,0)^2 = \varphi(z,0)$; hence,
\[
\hat \Xi(t,z,0) = \xi(z) \, \frac{\xi(z) \sinh(\xi(z) t) + \psi(z) \cosh(\xi(z) t)}{\psi(z) \sinh(\xi(z) t) + \xi(z) \cosh(\xi(z) t)}, \quad \psi \defeq \Psi(\cdot,0).
\]
Note that letting
\[
\call E(t,z;\zeta_1,\zeta_2) \defeq \zeta_1(z) \sinh(\xi(z) t) + \zeta_2(z) \cosh(\xi(z) t)
\]
one can write
\[
\hat \Xi(t,z,0) = \xi(z) \, \frac{\call E(t,z;\xi,\psi)}{\call E(t,z;\psi,\xi)} = \frac{\de}{\de t} \log \call E(t,z;\psi,\xi).
\]
For any $t \in I$ fixed, this function is well-defined for $z \in \varrho \bb D$ and analytic therein by the compatibility assumption~\eqref{hpcomp}. At this point, \eqref{dir_eqxi} becomes a first-order ODE in $t \in I$, for all $z,w \in \varrho \bb D$ fixed, whose solution is given by
\begin{equation} \label{dir_xitint}  \begin{split}
\hat \Xi(t,z,w) = \frac{\xi(z) \xi(w)}{\call E(t,z;\psi,\xi) \call E(t,w;\psi,\xi)} \bigg( &\Psi(z,w) + \int_0^t \Big( \call E(t,z;\xi,\psi) \call E(t,w;\xi,\psi) \\
&\ + \frac{\varphi(z,w)}{\xi(z) \xi(w)} \call E(t,z;\psi,\xi) \call E(t,w;\psi,\xi) \Big) \bigg).
\end{split} \end{equation}
Note that $\hat \Xi(t,\cdot,\cdot)$ is well-defined for $z,w \in \varrho \bb D$ by the same argument we applied to $\hat \Xi(t,\cdot,0)$. Also, it is trivial that $\hat \Xi(\cdot,z,w) \in C^\infty(I)$, and by differentiating under the integral sign one proves the analyticity of $\hat \Xi(t,\cdot,\cdot)$.
\end{proof}

\begin{rmk} \label{LQ_computeintegral}
Let $\tilde{\call E}(t,z;\zeta_1,\zeta_2) \defeq {\call E}(t,z;\zeta_2,\zeta_1)$. As, omitting the dependence on $\zeta_1$ and $\zeta_2$ in $\call E$, we have $\frac{\de}{\de t} \call E(t,z) = \xi(z) \tilde{\call E}(t,z)$, it is easy to see that
\[
\frac\de{\de t} \big( \tilde{\call E}(t,z) \call E(t,w) \pm \call E(t,z) \tilde{\call E}(t,w) \big) = (\xi(z)\pm\xi(w)) \big( \call E(t,z) \call E(t,w) \pm \tilde{\call E}(t,z) \tilde{\call E}(t,w) \big)
\]
and thus
\[
\int_0^t \call E(s,z) \call E(s,w)\,\di s = \frac12 \Big( \frac{\tilde{\call E}(\cdot,z) \call E(\cdot,w) + \call E(\cdot,z) \tilde{\call E}(\cdot,w)}{\xi(z)+\xi(w)} + \frac{\tilde{\call E}(\cdot,z) \call E(\cdot,w) - \call E(\cdot,z) \tilde{\call E}(\cdot,w)}{\xi(z)-\xi(w)} \Big)\bigg|_0^t.
\]
Letting $(\zeta_1,\zeta_2) \in \{(\xi,\psi),(\psi,\xi)\}$ one can then compute the integral in \eqref{dir_xitint}. Set
\begin{equation} \label{dir_spm}
\sigma^\pm(t,z,w) = \frac{\call L(t,z) + \call L(t,w)}{\xi(z) + \xi(w)} \pm \frac{\call L(t,z) - \call L(t,w)}{\xi(z) - \xi(w)}, \qquad \call L(t,\cdot) \defeq \frac{\call E(t,\cdot;\xi,\psi)}{\tilde{\call E}(t,\cdot;\xi,\psi)} ;
\end{equation}
then
\begin{equation} \label{dir_Xit}
2\,\hat\Xi(t,z,w) = \tilde\Psi(z,w) + \varphi(z,w) \sigma^+(t,z,w) + \xi(z)\xi(w) \sigma^-(t,z,w),
\end{equation}
where
\[
\tilde \Psi (z,w) = \Psi(z,w) - \varphi(z,w) \sigma^+(0,z,w) + \xi(z)\xi(w) \sigma^-(0,z,w).
\]
\end{rmk}

\begin{lem} \label{dir_lem:cl1}
Let $c=(c_{hk})_{h,k\in\N}$ be the solution to system~\eqref{dir_eqcz} on $[0,T]$. Then, for any $r \in (1,\varrho)$,
\begin{equation} \label{dir_estchkL}
\lVert c_{hk} \rVert_{\infty;[0,T]} \leq \frac{K(r,T)}{r^{h+k}} \quad \forall\, h,k \in \N,
\end{equation}
for some constant $K(r,T)$ depending only on $T$, $r$, $f$ and $g$. In particular, $c \in C^0([0,T]; \ell^1(\N^2))$ and the same is true for $\dot c$.\footnote{And for all higher order derivatives.}
\end{lem}

\begin{proof}
By the strong gathering assumption, the function $\hat \Xi(t,\cdot,\cdot)$ given in \Cref{dir_lem:xi} is analytic in a neighborhood of $\call Q_r \defeq \barr{\bb D}_r \times \barr{\bb D}_r$, with $r \in (1,\varrho)$. Then, by Cauchy's theorem on derivatives,
\begin{equation} \label{dir_estchk}
\lVert c_{hk} \rVert_{\infty;[0,T]} \leq \frac1{r^{h+k}} \max_{[0,T] \times \de \call Q_r} \lvert \hat\Xi \rvert.
\end{equation}
This proves that $c \in C^0([0,T]; \ell^1(\N^2))$, and further regularity is easily proven by induction by exploiting system~\eqref{dir_eqcz}.
\end{proof}

\begin{lem} \label{dir_lem:bf}
Let $c$ be the solution to~\eqref{dir_eqcz} on $[0,T]$. The functions $u^i$ given by \eqref{dir_uLQl} are well-defined for $(t,x) \in [0,T] \times \call X$, differentiable with respect to $t$, and twice Fréchet-differentiable with respect to $x$.
\end{lem}

\begin{proof}
Since a suitable shift of coordinates in $\call X$ transforms $u^i(t,\cdot)$ into $u^0(t,\cdot)$, it is sufficient to prove the result for $u^0$. The well-posedness follows from \Cref{dir_lem:cl1}, and the differentiability with respect to $t$ follows from the same lemma and the fundamental theorem of calculus. The differentiability with respect to $x$ follows again from \Cref{dir_lem:cl1}, as it is trivial to see that for any $h \in \call X$ small one has
\[
u^0(t,x+h) - u^0(t,x) - \pair{(c(t) \otimes I_d)x}{h}_{\call X} - \frac12 \pair{(c(t) \otimes I_d)h}{h}_{\call X} = 0,
\]
where, with an obvious notation, $\pair{(c(t) \otimes I_d)y}{z}_{\call X} = \sum_{h,k \in \N} c_{hk}(t) \pair{y^h}{z^k}_{\R^d}$.
\end{proof}

\begin{lem}
Let $u^i$ be given by \eqref{dir_uLQl}, where $c$ is the solution to \eqref{dir_eqcz} on $[0,T]$. Let $\call B_R$ be the closed ball of radius $R$ in $\call X$. Then, for any $r \in (1,\varrho)$,
\[
\sum_{j \in \Z \setminus \{i\}} \left\lVert D_ju^j D_ju^i \right\rVert_{\infty;[0,T]\times \call B_R} \leq \frac1{r-1} \bigg( \frac{R K r}{r-1} \bigg)^2,
\]
where $K=K(r,T)$ is the constant appearing in \Cref{dir_lem:cl1}.
\end{lem}

\begin{proof}
We have $D_j u^i(t,x) = 0$ if $j < i$, and $D_j u^i(t,x) = \sum_{h \in \N} c_{j-i,h}(t) x^h$ if $j \geq i$. Then, for $x \in \call B_R$, and $r \in (1,\varrho)$ fixed, estimate~\eqref{dir_estchkL} yields
\[
\left\lvert D_ju^j(t,x) D_ju^i(t,x) \right\rvert \leq R^2 \sum_{h,k \in \N} \lvert c_{0h}(t) \rvert \lvert c_{j-i,k}(t) \rvert \leq \bigg( \frac{RKr}{r-1} \bigg)^2 \frac{1}{r^{j-i}},
\]
for all $t \in [0,T]$. The thesis now follows by computing $\sum_{j > i} r^{i-j}$.
\end{proof}

Since, by construction, choosing $u^i$ as in \eqref{dir_uLQl} satisfying \ref{dir_s1} and \ref{dir_s2} in \Cref{dir_def:cs} yields a classical solution, the proof of \Cref{dir_wp} is complete.

\subsection{Almost-optimal controls for the \texorpdfstring{$N$}{N}-player game}\label{s:almostnash}

We did not prove long-time well-posedness of the Nash system for the $N$-player game, with $N > \ell$ finite; nevertheless, in \Cref{coreNE} below we show that on the horizon $[0,T]$ the infinite-dimensional optimal control for the $i$-th player, given by $\bar\alpha^{*i}(t,X_t) \defeq \sum_{j\geq i} c_{0,j-i}(t) X_t^j$,  if suitably ``projected'' onto $(\R^d)^N$ provides an $\epsilon$-Nash equilibrium for the $N$-player game, with $\epsilon \to 0$ as $N \to \infty$. In claiming so, we consider classes of admissible controls in the sense of the following definition.

\begin{defn}\label{admissibleL}
Let $R,L \geq 0$. A control $\alpha \colon [0,T] \times (\R^d)^N \to (\R^d)^N$ in feedback form belongs to $\call A_{R,L}$ if, for all $t \in [0,T]$, $x,y \in (\R^d)^N$,
\[
 |\alpha(t,0)| \leq R, \qquad |\alpha(t,x) - \alpha(t,y)| \leq L|x-y|.
\]
The Lipschitz constant $L$ is said to be \emph{admissible} if $L \geq \norm{c_{0\cdot}}_{C^0([0,T];\ell^1(\N))}$.
\end{defn}

\begin{rmk} \label{rmksolvSDE}
For such controls, it is known (cf., e.g., \cite[Theorems~9.1 and~9.2]{baldi}) that
\begin{equation} \label{SDE_dN}
\begin{cases}
\di X_t = \alpha(t,X_t)\,\di t + \sqrt2 \,\di B_t & t \in [0,T]\\
X_0 = x_0 \in (\R^d)^N
\end{cases}
\end{equation}
has a unique solution, satisfying $\bb E\,\sup_{[0,T]} |X|^2 \leq C$ where $C$ is a locally bounded function of $R$, $L$ and $T$, directly proportional to $1+|x_0|^2$.
\end{rmk}

\begin{thm} \label{LQ_eNE}
Consider the $N$-player game on $[0,T]$ with time evolution of the state of the players given by \eqref{SDE_dN} and costs given by \eqref{dir_costi} with $f^0,g^0 \geq 0$. Let Assumptions \textnormal{(\hyperlink{AS}{$\star$})} be in force with $[0,T] \subseteq I$.
Let $c$ solve \eqref{dir_eqcz} and define the control $\alpha^*$ by
\[
- \alpha^{*i}(t,X^0_t,\dots,X^{N-1}_t) \defeq \sum_{j=0}^{N-i-1} c_{0j}(t) X^{j+i}_t + \sum_{j=N-i}^{N-1} c_{0j}(t) X^{j+i-N}_t, \qquad i \in [[N]].
\]
Then, for any $R \geq 0$ and admissible $L$, for any $i \in [[N]]$ and any $(\alpha^{*,-i},\psi) \in \call A_{R,L}$,\footnote{With this notation we mean that all components are those of $\alpha^*$ but the $i$-th one, which is a suitable $\R^d$-valued function $\psi = \psi(t,x)$. Note that $\alpha^* \in \call A_{R,L}$ for any $R$ and admissible $L$.}
\[
J^i(\alpha^*) \leq J^i((\alpha^{*,-i},\psi)) + \hat C(\delta^M + (\delta^{-M} + N ) \delta^N) \qquad \forall\, \delta \in (\varrho^{-1},1), \ \forall\;\! M \geq \ell,
\]
where $\ell$ is the dimension appearing in Assumptions \textnormal{(\hyperlink{AS}{$\star$})} and the constant $\hat C$ is a locally bounded function of $R$, $L$, $T$ and $\delta$, directly proportional to $1+|x_0|^2$.
\end{thm}

\begin{proof}
Since $f^i$ and $g^i$ are shift-invariant and $\alpha^*$ is linear in the state variable with $D\alpha^*$ circulant, without loss of generality we can prove the thesis for $i = 0$. We denote by $X^*$ the solution to \eqref{SDE_dN} when $\alpha = \alpha^*$ and by $X$ the solution to \eqref{SDE_dN} when $\alpha = (\psi, \alpha^{*1},\dots,\alpha^{*,N-1}) \eqdef \hat\alpha^*$. We wish to estimate from above the quantity
\begin{equation} \label{diffJ} \begin{multlined}[t][.85\displaywidth]
J^0(\alpha^*) - J^0(\hat\alpha^*) = \frac12 \, \bb E\bigg[ \int_0^T \!\! \big( |\alpha^{*0}(t,X^*_t)|^2 - |\psi(t,X_t)|^2 + F(X^*_t) - F(X_t) \big)\,\di t \\
+ G(X^*_T) - G(X_T) \bigg],
\end{multlined} \end{equation}
where $F \defeq \pair{F^0\cdot}{\cdot}$ and $G \defeq \pair{G^0\cdot}{\cdot}$.
Let $u$ be the QSD solution to \eqref{dir_NS} on $[0,T] \times \call X$. Since $G$ is convex and $DG = Du^0(T,\cdot)$,
\[\begin{split}
G(X^*_T) - G(X_T) &\leq \pair{Du^0(T,X^*_T)}{X^*_T - X_T}_{\R^{dN}} = \sum_{j=0}^{N-1} \sum_{k = 0}^{\ell-1} c_{jk}(T) \pair{X^{*k}_T}{X^{*j}_T-X^j_T}_{\R^d} \\
&= \sum_{j=0}^{N-1} \sum_{k = 0}^{\ell-1} \bigg( \int_0^T \dot c_{jk}(t) \pair{X^{*k}_t}{X^{*j}_t-X^j_t}_{\R^d}\, \di t \\
&\quad\ + c_{jk}(t) \pair{\di X^{*k}_t}{X^{*j}_t-X^j_t}_{\R^d} + c_{jk}(t) \pair{X^{*k}_t}{\di(X^{*j}_t-X^j_t)}_{\R^d} \bigg).
\end{split}\]
Note that we can replace $\ell$ with any $M \geq \ell$ because $c_{jk}(T) = g_{jk}(T) = 0$ if $k > \ell$. 
As $c$ solves \eqref{dir_eqcz} we obtain
\[
\begin{split}
G(X^*_T) - G(X_T) &\leq \sum_{j=0}^{N-1} \sum_{k = 0}^{M-1} \bigg( \int_0^T \sum_{h = 1}^{j} c_{0,j-h}(t) c_{hk}(t) \pair{X^{*k}_t}{X^{*j}_t-X^j_t}_{\R^d} \, \di t \\
&\quad + \int_0^T \sum_{h = 0}^{k} c_{0,k-h}(t) c_{hj}(t) \pair{X^{*k}_t}{X^{*j}_t-X^j_t}_{\R^d} \,\di t \\
&\quad - f_{jk} \int_0^T \pair{X^{*k}_t}{X^{*j}_t-X^j_t}_{\R^d}\,\di t \\
&\quad+ \int_0^T c_{jk}(t) \pair{\alpha^{*k}(t,X^*_t)}{X^{*j}_t-X^j_t}_{\R^d} \,\di t \\
&\quad + \int_0^T c_{jk}(t) \pair{X^{*k}_t}{\alpha^{*j}(t,X^*_t)-\hat\alpha^{*j}(t,X_t)}_{\R^d} \,\di t \bigg) + Z_T,
\end{split}
\]
where $(Z_t)_{0\leq t\leq T}$ is a martingale starting from $0$. Straightforward computations show that, omitting the dependence on $t$,
\[
- \sum_{k=0}^{M-1} c_{jk} \alpha^{*k}(X^*) = \sum_{k=0}^{M-1}\sum_{h=0}^{k} c_{0,k-h} c_{hj} {X^{*k}} + \sum_{h=N-M+1}^{N-1} \sum_{k=N-h}^{M-1} c_{0h} c_{jk} X^{*,h+k-N}
\]
and
\[ \begin{split}
- \sum_{j=0}^{N-1} c_{jk}(\alpha^{*j}(X^*)-\hat\alpha^{*j}(X)) &= - c_{0k}(\alpha^{*0}(X^*) - \psi(X)) + \sum_{j=0}^{N-1} \sum_{h=1}^{j} c_{hk}c_{0,j-h} (X^{*j}-X^j) \\
&\quad\ + \sum_{h=N}^{2N-2} \sum_{j=h-N+1}^{N-1} c_{jk} c_{0,h-j} (X^{*,h-N} - X^{h-N});
\end{split}\]
therefore,
\begin{equation} \label{estDeltaG} \begin{split}
G(X^*_T) - G(X_T) &\leq - \sum_{k=0}^{M-1} \int_0^T c_{0k}(t) \pair{X^{*k}_t}{\alpha^{*0}(X^*) - \psi(X)}_{\R^d} \, \di t \\
&\quad\ - \sum_{j=0}^{N-1} \sum_{k = 0}^{M-1} f_{jk} \int_0^T \pair{X^{*k}_t}{X^{*j}_t-X^j_t}_{\R^d}\,\di t - \int_0^T \call E_t \,\di t,
\end{split}\end{equation}
where we have set
\[  \begin{split}
\call E &\defeq \sum_{j=0}^{N-1} \sum_{h=N-M+1}^{N-1} \sum_{k=N-h}^{M-1} c_{0h} c_{jk} \pair{X^{*,h+k-N}}{X^{*j}-X^j}_{\R^d} \\
&\quad\ + \sum_{k=0}^{M-1} \sum_{h=N}^{2N-2} \sum_{j=h-N+1}^{N-1} c_{jk} c_{0,h-j} \pair{X^{*k}}{X^{*,h-N} - X^{h-N}}_{\R^d}.
\end{split}\]
Note now that by the convexity of $\frac12|\cdot|^2$, omitting the dependence on $t$,
\begin{equation} \label{conv12n}
\frac12 |\alpha^{*0}(X^*)|^2 - \frac12 |\psi(X)|^2 - \pair{\alpha^{*0}(X^*)}{\alpha^{*0}(X^*) - \psi(X)}_{\R^d} \leq 0
\end{equation}
and by the convexity of $F$
\begin{equation} \label{proofconvF}
F(X^*) - F(X) - \pair{DF(X^*)}{X^*-X}_{\R^d} \leq 0.
\end{equation}
Using \eqref{estDeltaG}, \eqref{conv12n} and \eqref{proofconvF} in \eqref{diffJ} we obatin
\begin{equation} \label{estJJ1}
J^0(\alpha^*) - J^0(\hat\alpha^*) \leq \bb E \bigg[ \sum_{k=\ell}^{N-1} \int_0^T c_{0k}(t) \pair{X^{*k}_t}{\alpha^{*0}(X^*) - \psi(X)}_{\R^d} \, \di t - \int_0^T \call E_t \,\di t \bigg].
\end{equation}
As, whenever $\{ Y,Z \} \subseteq \{ X^*,X \}$,
\[
\sup_{j,k \in [[N]]} \bb E \int_0^T |\pair{Y^{k}_t}{Z^{j}_t}_{\R^d}| \,\di t \leq T C,
\]
where $C = C(|x_0|,R,L,T)$ is the constant appearing in \Cref{rmksolvSDE}, we have
\[ \begin{split}
\bigg| \,\bb E \int_0^T \call E_t \, \di t\, \bigg| &\leq 2TC \bigg( \norm{c}_{C^0([0,T];\ell^1(\N^2))} \sum_{h \geq N-M} \norm{c_{0h}}_{\infty;[0,T]} \\
&\quad\ + \sum_{h\geq N} \sum_{j=0}^{N-1} \norm{c_{j\cdot}}_{C^0([0,T];\ell^1(\N))} \norm{c_{0,h-j}}_{\infty;[0,T]} \bigg),
\end{split}\]
so that by \Cref{dir_lem:cl1}, for any $\delta \in (\varrho^{-1},1)$, there exists $K > 0$ such that
\begin{equation} \label{estEiE}
\bigg| \,\bb E \int_0^T \call E_t \, \di t\, \bigg| \leq \tilde C ( \delta^{-M} + N ) \delta^N, \quad \tilde C \defeq \frac{2TCK^2}{(1-\delta)^3}.
\end{equation}
On the other hand,
\begin{equation} \label{estprimap}
\sup_k \, \bb E \int_0^T |\pair{X^{*k}_t}{\alpha^{*0}(X^*) - \psi(X)}_{\R^d}| \leq 2R\sqrt{TC} + 2LC,
\end{equation}
where we used again that $\alpha^*, \hat\alpha^* \in \call A_{R,L}$. Therefore, from \eqref{estJJ1}, \eqref{estEiE}, \eqref{estprimap} and \Cref{dir_lem:cl1} we obtain
\[
J^0(\alpha^*) - J^0(\hat\alpha^*) \leq  \hat C (\delta^M + (\delta^{-M} + N ) \delta^N), \quad \hat C \defeq 2(R\sqrt{TC} + LC)\frac{K}{1-\delta} \vee \tilde C.
\]
This concludes the proof.
\end{proof}

\begin{cor} \label{coreNE}
Let $\Omega \subset \R^d$ bounded, and assume the same as in \Cref{LQ_eNE}, but with $x_0 \in \Omega^N$. Let $R \geq 0$ and $L$ be admissible (in the sense of Definition \ref{admissibleL}); consider as admissible those controls belonging to $\call A_{R,L}$. Then, for any $\epsilon > 0$ there exists $N_0 = N_0(\epsilon,R,L,\varrho, \ell, \Omega)$ such that if $N \geq N_0$ then the control $\alpha^*$ provides an $\epsilon$-Nash equilibrium of the game.
\end{cor}

\begin{proof}
It suffices to require that
\begin{equation} \label{condforeps}
\hat C (\delta^M + (\delta^{-M} + N ) \delta^N) \leq \epsilon \quad \forall\, N \geq N_0,
\end{equation}
where, for instance, one sets $\delta = \frac12(1+\varrho^{-1})$. Choose $M = M(N)$ such that $M \to \infty$ and $M = o(N)$ as $N \to \infty$; for example, $M = \lfloor \sqrt{N} \rfloor$ for all $N > \ell^2$. Since $x_0 \in \Omega^N$ there exists a constant $\hat C'$, independent of $N$, such that $\hat C \leq N \hat C'$. Then the conclusion follows from the fact that the left-hand side of \eqref{condforeps} goes to $0$ as $N \to \infty$.
\end{proof}

\subsection{The ergodic Nash system}

Consider now the related ergodic problem, with costs
\[
\bar J^i(\alpha) = \liminf_{T \to +\infty} \,\frac1{2T}\, \bb E \int_0^T \!\big( \lvert \alpha^i \rvert^2 + \pair{F^i X}{X} \big),
\]
where the dynamics and the assumptions on $F^i$ are the same as before. The corresponding Nash system reads
\begin{equation} \label{dir_NSe}
\lambda^i - \Delta v^i + \frac12 \lvert D_i v^i \rvert^2 + \sum_{j\neq i} D_j v^j D_j v^i = \bar F^i \qquad \text{on} \ \call Y,
\end{equation}
where $\call Y$ is either $(\R^d)^N$, for the $N$-player game, or $\call X$, for the limit game with infinitely many players. In the latter case we give notions of classical solution and QSD solution which are analogous as those in the previous section.

\begin{defn} \label{dir_def:cse}
A sequence of pairs $((\lambda^i, v^i))_{i\in\Z}$ of real numbers $\lambda^i$ and $\R$-valued functions $v^i$ defined on $\call X$ is a classical solution to the ergodic Nash system~\eqref{dir_NSe} on $\call X$ if the following hold:
\begin{enumerate}[label=\textbf{(E\arabic*)}]
\item \label{dir_e1} each $v^i$ is of class $C^2$ with respect to $x \in \call X$, in the Fréchet sense;
\item \label{dir_e2} for each $i \in \N$, the series $\sum_{j\neq i} D_j v^j D_j v^i$ uniformly converges on all bounded subsets of $\call X$;
\item \label{dir_e3} system~\eqref{dir_NSe} is satisfied pointwise for all $x \in \call X$;
\end{enumerate}
A \emph{QSD ergodic (QSDE) solution} will be a classical solution to \eqref{dir_NSe} with
\begin{equation} \label{dir_defQSDE}
\lambda^i \equiv \lambda = \tr \bar c, \qquad v^i(x) = \frac12 \sum_{h,k \in \N} \bar c_{hk} \pair{x^{h+i}}{x^{k+i}}_{\R^d},
\end{equation}
for some $\bar c \in \ell^1(\N^2)$.
\end{defn}

\begin{rmk}
By the structure of \eqref{dir_NSe}, it is clear that if $((\lambda^i, v^i))_{i\in\Z}$ is a classical solution, then so will be $((\lambda^i, v^i+\mu^i))_{i\in\Z}$ for any choice of real numbers $\mu^i$. We will prove that there exists a special choice of $\mu \in \R$, and of $\bar c \in \ell^1(\N^2)$, such that the solution $((\lambda, v^i+\mu))_{i\in\Z}$, with $\lambda$ and $v^i$ given by \eqref{dir_defQSDE}, is in a precise sense the limit of the QSD solution as $T \to +\infty$ (see~\Cref{dir_thm:conv}).
\end{rmk}

Arguing as in the previous section, the coefficients $\bar c_{hk}$ of a QSDE solution are given by the solutions of the following system:
\begin{equation} \label{dir_eqcze}
-   c_{0h}   c_{0k} + \sum_{j = 0}^h   c_{0,h-j}   c_{jk} + \sum_{j = 0}^k   c_{0,k-j}   c_{jh} = f_{hk}.
\end{equation}
It is immediate to see that if $ c$ solves \eqref{dir_eqcze}, then so does $- c$, hence we cannot have a unique solution to this limit system.

\begin{lem} \label{dir_lem:solce}
There are exactly two sequences $(c_{hk}^\pm)_{h,k \in \N}$ which solve \eqref{dir_eqcze}. Such sequences are one the opposite of the other; that is, $c^- = - c^+$.
\end{lem}

\begin{proof}
We have $c_{00}^2 = f_{00} > 0$, hence $c_{00} \in \{ \pm \sqrt{f_{00}}\, \}$. Once the sign of $c_{00}$ is chosen, all other $c_{hk}$ can be uniquely determined by induction on $h,k$.
\end{proof}

Both solutions can be represented also in this case by a generating function.

\begin{lem} \label{dir_lem:bxi}
The two solutions to \eqref{dir_eqcze} are given by
\begin{equation} \label{dir_cderbxi}
c_{hk}^\pm = \pm \,\frac1{h!\,k!} \, \frac{\de^{h+k}}{\de z^h \de w^k}\bigg|_{(0,0)} \barr\Xi,
\end{equation}
for some analytic function $\barr\Xi \colon \varrho \bb D^2 \to \C$ such that $\barr\Xi(z,w) = \barr\Xi(w,z)$.
\end{lem}

\begin{proof}
A function $\barr\Xi$ is the desired generating function if it solves the equation
\begin{equation*} 
- \barr\Xi(z,0)\,\barr\Xi(0,w) + \left( \barr\Xi(z,0) + \barr\Xi(0,w) \right) \barr\Xi(z,w) = \varphi(z,w),
\end{equation*}
for $z,w \in \varrho \bb D$. It follows that $\barr\Xi(z,0)^2 = \varphi(z,0)$, so choose $\barr\Xi(\cdot,0) = \xi$; then
\begin{equation} \label{dir_barXi}
\barr \Xi(z,w) = \frac{\varphi(z,w) + {\xi(z)}{\xi(w)}}{{\xi(z)} + {\xi(w)}}.
\end{equation}
Note that the real part of the denominator in \eqref{dir_barXi} can vanish only if $|z|,|w| = \varrho$, hence $\barr\Xi$ is well-defined and analytic for $(z,w) \in \varrho \bb D^2$.
\end{proof}

This is sufficient to prove the existence of exactly two opposite QSDE solutions, and thus that the limiting ergodic Nash system is well-posed in $\call X$. We state this as a theorem.

\begin{thm}
There exist exactly two QSDE solutions to \eqref{dir_NSe} on $\call X$. Such solutions are one the opposite of the other.
\end{thm}

\begin{proof}
Argue as in the first part of the proof of \Cref{dir_lem:cl1} to say that $c^+ \in \ell^1(\N^2)$, then argue as in the proof of \Cref{dir_lem:bf} to build the QSDE solution determined by choosing $\bar c = c^+$. Finally note that the solution determined by the choice $\bar c = c^-$ is the opposite function.
\end{proof}

\subsection{Long-time asymptotics}

As expected by KAM theory and ergodic control, we are going to prove now that, up to a constant, the QSDE solution corresponding to $c^+$ (that is, the solution with $c_{00} > 0$) describes the long-time asymptotics of the QSD solution as $T \to +\infty$, while the opposite solution should be regarded as the result of considering the limit $T \to -\infty$ instead. To highlight the dependence of the QSD solution on $T$, we will write it as $u^i_T$; on the other hand, since by the shift-invariance property it will suffice to show the convergence of $u^0_T$ to the QSDE solution $v^0$, we will omit the superscript $0$.

\begin{thm} \label{dir_thm:conv}
Let Assumptions \textnormal{(\hyperlink{AS}{$\star$})} be in force with $[0,+\infty) \subseteq I$. Let $u_T$ be the value function of the $0$-th player corresponding to the QSD solution to the Nash system on $[0,T] \times \call X$; let $v$ be the value function of the $0$-th player corresponding to the QSDE solution on $\call X$ determined by $\bar c = c^+$. Let $\lambda \defeq \tr \bar c$. Then, for any $t \geq 0$, as $t < T \to +\infty$, the following limits hold, locally uniformly in both $x \in \call X$ and $t$:
\begin{equation} \label{dir_conv1}
\frac{u_T(t,x)}{T-t} \to \lambda
\end{equation}
and there exists a constant $\mu \in \R$ such that
\begin{equation} \label{dir_conv2}
u_T(t,x) - \lambda(T-t) \to v(x) + \mu.
\end{equation}
\end{thm}

The proof is based on the following result, which is strictly related to a refinement of \Cref{dir_lem:cl1} (cf.~\Cref{rmk:refLT} below) and is due to the possibility of explicitly compute the integral in formula~\eqref{dir_xitint} as showed in \Cref{LQ_computeintegral}.

\begin{lem} \label{lem:convE}
Under Assumptions \textnormal{(\hyperlink{AS}{$\star$})} with $[0,+\infty) \subseteq I$, there exists a nonnegative function $\gamma \in C_0^0([0,+\infty)) \cap L^1([0,+\infty))$, depending only on $r$, $f$ and $g$, such that 
\begin{equation*} 
\sup_{(z,w) \in \call Q_r} \bigg\lvert \sigma^\pm(t,z,w) - \frac{2}{\xi(z) + \xi(w)} \bigg\rvert \leq \gamma(t)\,
\end{equation*}
for all $t \geq 0$, where $\sigma^\pm$ are defined as in~\eqref{dir_spm}.
\end{lem}

\begin{proof}
By the continuity of $\xi$, given $r' \in (r,\varrho)$,\footnote{E.g., $r' = \frac12(r+\varrho)$.}  there exists $\epsilon > 0$ such that
\begin{equation} \label{dir_rexi}
\Re \,\xi \geq \epsilon \quad \text{on}\ \barr{\bb D}_{r'}
\end{equation}
whence
\begin{equation} \label{dir_infde}
\inf_{k\in\Z} \,\lvert{\xi(w)t - \big(k+\tfrac12\big)i\pi}\rvert \geq \frk d(t) \qquad \forall\, w \in \barr{\bb D}_{r'}, \ t \in [0,+\infty),
\end{equation}
for some function $\frk d$ which is uniformly positive on $[0,+\infty)$.\footnote{For instance, $\frk d(t) = \big( t \vee \tfrac13\norm{\xi}_{\infty;\barr{\bb D}_{r'}}^{-2} \big)\epsilon$.}
By \eqref{dir_infde} there exists a uniformly positive function ${\frk f} \colon [0,+\infty) \to \R_+$, which depends only on $r$ and $f$, such that $\lvert \cosh(\cdot\,t) \rvert \geq {\frk f}(t)$ on $\xi(\barr{\bb D}_r)$; also, by \eqref{dir_rexi} we can suppose that ${\frk f}$ be asymptotic to $\frac12 e^{3\epsilon|\cdot|}$ at $\infty$. Since, with $\call L$ defined as in \eqref{dir_spm},
\[
\frac{\de}{\de\xi} \, \call L(t,z;\xi(z),\psi(z)) = \frac1{\cosh(\xi(z)t)} \, \frac{(\xi(z)^2-\psi(z)^2)t - \frac{\psi(z)}{\cosh(\xi(z)t)}}{(\psi(z)\tanh(\xi(z) t)+\xi(z))^2},
\]
we obtain, also using \eqref{hpcomp},
\begin{equation} \label{dir_s-bound}
\bigg| \frac{\call L(t,z) - \call L(t,w)}{\xi(z) - \xi(w)} \bigg| \lesssim \frac{t}{\frk f(t)^2} \quad \forall \, (t,z,w) \in [0,+\infty) \times \call Q_r,
\end{equation}
where the implied constant depends only on $r$, $f$ and $g$. At this point, it is easy to see that the desired conclusion follows.
\end{proof}

\begin{rmk} \label{rmk:refLT}
This proof shows that if $[0,+\infty) \subseteq I$, then the constant $K$ appearing in \Cref{dir_lem:cl1} is in fact independent of $T$, since $\sup_{\R \times \call Q_r} \lvert \hat\Xi \rvert$ is finite. In particular, $c \in C^0([0,T]; \ell^1(\N^2))$ along with its derivatives, and their the norms are bounded uniformly with respect to $T>0$.
\end{rmk}

\begin{proof}[Proof of \Cref{dir_thm:conv}]
Fix $r \in (1,\varrho)$. By comparing expressions~\eqref{dir_Xit} and \eqref{dir_barXi} one sees that
\begin{equation} \label{dir_xitobxi}
\sup_{|z|,|w| \leq r} \lvert \Xi(t,z,w) - \barr\Xi (z,w) \rvert \lesssim \gamma(T-t),
\end{equation}
where $\gamma$ is the function given in \Cref{lem:convE}, and the implied constant depends only on the $L^\infty$-norms of $\varphi$ and $\xi$ on $\bb D_r$ and $\call Q_r$, respectively.
By \Cref{dir_lem:xi} and Cauchy's theorem on derivatives,
\begin{equation} \label{dir_ctobc}
\lvert c_{hk}^T(t) - \bar c_{hk} \rvert \leq \frac1{r^{h+k}} \sup_{|z|,|w| \leq r} \lvert \Xi(t,z,w) - \barr\Xi (z,w) \rvert \qquad \forall\, h,k \in \N;
\end{equation}
where we have used the superscript $T$ to stress the fact that $c_{hk}(t) = c_{hk}^T(t)$ depends on the horizon $T$. Plugging \eqref{dir_xitobxi} into \eqref{dir_ctobc} yields
\begin{equation} \label{dir_convsums}
\sum_{h,k\in\N} \lvert c_{hk}^T(t) - \bar c_{hk} \rvert \lesssim \gamma(T-t)
\end{equation}
as well as
\begin{equation} \label{dir_convtrs}
\lvert \,\tr c^T(t) - \lambda\, \rvert \lesssim \gamma(T-t),
\end{equation}
where the implied constants depend only on $r$ and $f$. 
As $\gamma$ is integrable on $[0,+\infty)$, by \eqref{dir_convtrs} so is $\tr \hat c - \lambda$, where we use the notation $\hat c = c(T-\cdot)$ introduced in the proof of \Cref{dir_lem:uniquesolc}; thus by the dominated convergence theorem there exists $\mu \in \R$ such that
\[
\int_t^T \tr c^T - (T-t)\lambda \to \mu \quad \text{as} \ T\to+\infty,
\]
locally uniformly in $t$. Along with \eqref{dir_convsums}, this proves \eqref{dir_conv1} and \eqref{dir_conv2}.
\end{proof}

\begin{rmk}
The argument of the previous proof also applies to the case when $t = sT$, with $s \in [0,1]$. In this case, we can give the following estimate of the rate of convergence of \eqref{dir_conv1}: for any $r \in (1,\varrho)$,
\[
\sup_{\norm{x}_{\call X} \leq L, \ s \in [0,1]} \bigg\lvert \frac{u_T(sT,x)}{T} - (1-s)\lambda \bigg\rvert \lesssim_{L,r} \frac1T.
\]
The implied constant can be computed quite explicitly, by retracing the proofs above; we confine ourselves to pointing out that it depends only on $L$, $r$ and $f$, and that it explodes as $L \to \infty$ or $r \to 1$.
\end{rmk}

\subsection{Digression on a delicate limit case} \label{dir_sec:lc}

We have noted in \Cref{dir_ex1} that the case of a cost designed according to an underlying directed circulant graph structure is limit among those satisfying our assumptions, in the sense that \Cref{def:sg} holds with $\varrho = 1$.

Results like Lemmata \ref{dir_lem:uniquesolc}, \ref{dir_lem:xi}, \ref{dir_lem:cl1}, \ref{dir_lem:solce} and \ref{dir_lem:bxi} continue to hold, but the other methods used in the previous sections are not refined enough to successfully prove all the previous theorems for those limit cases, even though, along with \Cref{rmk:refLT}, they are sufficient in order to establish $\ell^\infty$-stability at the level of the system for $c$; that is, convergence in $\ell^\infty(\N^2)$ of the solution to \eqref{dir_eqcz} to a solution of \eqref{dir_eqcze}.

On the other hand, having a closer look at the simplest limit case, which is the directed chain given by the choice $g = 0$, $f^0_{00} = f^0_{11} = 1 = -f^0_{01} = -f^0_{10}$ and $f^0_{hk} = 0$ for all other $h,k$, we note that we are also able to compute QSDE solutions quite easily, thanks to formula \eqref{dir_barXi}. In this case we have $\varphi(z,0) = 1 - z$, and we find the expansion
\[
\barr \Xi(z,w) = 1 + \sum_{j \geq 1} (-)^j \binom{\frac12}{j} \big( z^j + w^j \big) + \sum_{j \geq 2} (-)^j \binom{\frac32}{j} \sum_{\substack{h,k \geq 1 \\ h+k = j}} z^hw^k,
\]
yielding
\begin{equation} \label{dir_cexp}
c^\pm_{hk} = \pm(-)^{h+k}
\binom{\frac32-\delta_{0,hk}}{h+k}.
\end{equation}
As by Stirling's formula $\lvert c^\pm_{hk} \rvert \asymp (h+k)^{\delta_{0,hk}-\frac52}$, one easily sees that these coefficients well-define a QSDE solution, hence the limit ergodic Nash system is well-posed.

One can also note that the coefficients enjoy the property that
\begin{equation} \label{dir_c-c}
\bar c_{hk} = \bar c_{0,h+k} - \bar c_{0,h+k-1} \quad \text{if} \ hk \neq 0,
\end{equation}
where $\bar c$ is either $c^+$ or $c^-$; this can be seen from \eqref{dir_cexp} or proved by induction using system~\eqref{dir_eqcze}. In fact, the same can be done for system~\eqref{dir_eqcz}, so that property~\eqref{dir_c-c} also holds for the\footnote{Note that it is indeed unique, as \Cref{dir_lem:uniquesolc} is still true with the same proof.} solution of \eqref{dir_eqcz} on $[0,T]$, for any fixed $T>0$. Therefore, the information about the coefficients of a prospective QSD or QSDE solution is encoded in the functions of one complex variable $\Xi_0(t,\cdot) \defeq \Xi(t,\cdot,0)$ and $\barr\Xi_0 = \barr\Xi(\cdot,0)$, namely
\[
\Xi_0(t,z) = \sqrt{1-z}\, \tanh\left(\sqrt{1-z}\,(T-t)\right) \quad \text{and} \quad \barr\Xi_0(z) = \sqrt{1-z}\,.
\]

This peculiar fact is specific of the directed chain. It could be useful as it allows to conclude well-posedness of the infinite-dimensional evolutive Nash system provided that the sequence of functions $(c_{0k})_{k\in\N}$ is monotone,\footnote{As it would seem by computing the first functions of the sequence\dots} yet this would not still be enough to deal with convergence to an ergodic solution.

Another property is instead shared by all problems having $f$ of the form \eqref{dir_qtens}: the polynomial $\varphi$ factors as $\varphi(z,w) = \xi^2(z)\xi^2(w)$; this makes $\Xi$ and $\barr \Xi$ functions of $(\xi(z),\xi(w))$, possibly helping in the analysis of the aforementioned limit cases.

\section{Mean-field-like games}

In this part we consider the Nash system \eqref{dir_peqc} without the shift-invariant assumption. Instead, we make a \emph{mean-field-like} assumption. 
Given $(a^i_{jk})_{i,j,k} \in \Sym(N)^N$ we will use the notation $B(a)$ for the matrix $B(a)_{hk} = a^{h}_{hk}$. Given the $i$-th vector $e^i$ of the canonical basis of $\R^N$, we will write $E^i = e^i (e^i)^\trn$. The indices will range over $[[N]]$.

Note that system \eqref{dir_peqc} can be rewritten in a forward form as
\begin{equation} \label{MF_sys}
 \begin{cases}
\dot c^i - (c^i)^\trn E^ic^i + B(c)^\trn c^i + c^i B(c) = f^i \\
c^i(0) = g^i 
\end{cases}
\quad \ i \in [[N]],
\end{equation}
where $c^i(t), f^i(t),g^i \in \Sym(N)$ for all $i \in [[N]]$ and $t \in [0,T]$.

\begin{thm} \label{MF_thm}
Let $T>0$ and let $f \in L^1((0,T);\Sym(N)^N)$ and $g \in \Sym(N)^N$ satisfy
\begin{equation} \label{MF_hypfg}
\sup_i \Big( N \sum_{\substack{h,k\\ k \neq i}} | \bullet^i_{hk} |^2 + N \sum_{\substack{k \\ k \neq i}} | \bullet^k_{ki} |^2 +  |\bullet^i_{ii} |^2 \Big) \leq {\kappa_\bullet}, \quad \kappa_\bullet > 0, \quad \text{for}\ \bullet \in \{ f,g\}.
\end{equation}
Suppose that $B(f) \geq -K_f I$ and $B(g) \geq -K_g I$ for some constants $K_f$ and $K_g$ such that
\begin{equation} \label{MF_condKf}
K_g < \sup_{M\in \R} Me^{-2MT}, \qquad K_f < \sup_{M \in \R} \frac{2M(Me^{-MT}-K_g)}{1-e^{-2MT}}.
\end{equation}
Then there exists $N_0 \in \N$ such that if $N > N_0$ there exists a unique absolutely continuous solution $c$ to \eqref{MF_sys} on $[0,T)$, which satisfies
\begin{equation} \label{MF_estCN}
\sup_i \Big( N \sum_{\substack{h,k\\ k \neq i}} | c^i_{hk} |^2 + N \sum_{\substack{k \\ k \neq i}} | c^k_{ki} |^2 +  |c^i_{ii} |^2 \Big) \leq C
\end{equation}
for some constant $C$ which is independent of $N$.
\end{thm}

Note that \eqref{MF_condKf} is automatically satisfied when $K_f, K_g \leq 0$ for any $T > 0$. Otherwise, for fixed $K_f$ and $K_g$, it poses a restriction on the size of $T$ (or, similarly, a restriction on the size of $(K_f)^+$ and $(K_g)^+$ once $T$ is fixed).  Back to the value functions $u^i$, the previous estimate reads as follows:
\[
\sup_i \Big( N \sum_{\substack{h,k \\ k \neq i}} \| D^2_{hk}u^i \|_\infty^2 + N \sum_{\substack{k \\ k \neq i}} \| D^2_{ki}u^k \|_\infty^2 +  \| D^2_{ii}u^i \|_\infty^2 \Big) \leq C;
\]
in particular,
\[
\sup_i \sum_{j:\, j \neq i} \|D_j \alpha^i\|_\infty^2 = \sup_i \sum_{j:\, j \neq i} \|D^2_{ij} u^i\|_\infty^2 \leq \frac C N.
\]

\begin{rmk}\label{mfl} We use the terminology mean-field-like since any $f^i$ such that
\[
\sup_i |f^i_{ii}| + N \sup_{\substack{i,j \\ j\neq i}} |f^i_{ij}| + N^2 \!\!\! \sup_{\substack{i,j,k \\ j\neq i, k \neq i}} |f^i_{jk}| \leq C
\]
satisfies \eqref{MF_hypfg}, with $\kappa_f$ depending on $C$ (and not on $N$). In turn, the previous inequality is satisfied when $f^i(x) = V^i(x^i, (N-1)^{-1} \sum_{j \neq i} \delta_{x^j})$, where $V^i$ is a smooth enough function defined over $\R^d \times \Pc(\R^d)$ (see for instance \cite[Proposition 6.1.1]{CDLL}).
\end{rmk}

The proof of the theorem is based on the following lemmata.

\begin{lem} \label{MF_l1}
Let $c$ be an absolutely continuous solution to \eqref{MF_sys} with $B(c) \geq -MI$ on $[0,T)$ for some $M \in \R$. Assume \eqref{MF_hypfg}. Then the following estimates hold on $[0,T)$:
\begin{gather}
\label{MF_G1}
\sum_{\substack{h,k \\ k\neq i}} |c^i_{hk}|^2 \leq \kappa_0 N^{-1}, \qquad \kappa_0(t) \defeq (\kappa_g+ \kappa_f t)e^{(1+4M_+)t}, \\
 \label{MF_estcallS}
\sup_k \sum_{i\neq k} |c^i_{ik}|^2 \leq \kappa_1N^{-1}, \qquad k_1(t) \defeq 2\kappa_0(t) e^{2 \int_0^t\! \sqrt{\kappa_0}\,}, \\
\label{MF_estcallC0}
\sup_i |c^i_{ii}|^2 \leq \kappa_2, \qquad \kappa_2(t) = (\kappa_g + (\kappa_f + \kappa_0 \kappa_1 N^{-2}) t)e^{(2+M_+)t},
\end{gather}
where $M_+ = M \vee 0$.
As a consequence, $c$ continuously extends on $[0,T]$.
\end{lem}

\begin{proof}
Multiplying the equation for $c^i_{hk}$ by $c^i_{hk}$ and summing over $h$ and $k\neq i$ we have
\[ \begin{split}
\frac12 \frac{\di}{\di t} \sum_{\substack{h,k \\ k\neq i}} |c^i_{hk}|^2 &=  \sum_{\substack{h,k \\ k\neq i}} f^i_{hk} c^i_{hk} - \sum_{\substack{j,h,k \\ j,k\neq i}} c^i_{hk} c^i_{jh}  c^j_{jk} - \sum_{\substack{j,h,k \\ k\neq i}} c^i_{hk} c^i_{jk}  c^j_{jh} \\
&= \sum_{\substack{h,k \\ k\neq i}} f^i_{hk} c^i_{hk} - \mathrm{tr}(\hat c^i{}^\trn B( c) \hat c^i) - \mathrm{tr}(\tilde c^i{}^\trn B( c) \tilde c^i),
\end{split}
\]
where we have set $\hat c^i_{hk} \defeq c^i_{hk}(1-\delta_{hi})$ and $\tilde c^i_{hk} = c^i_{hk}(1-\delta_{ki})$. It follows that
\[
\frac{\di}{\di t} \sum_{\substack{h,k \\ k\neq i}} |c^i_{hk}|^2 \leq \kappa_f N^{-1} + (1+4M_+) \sum_{\substack{h,k \\ k\neq i}} |c^i_{hk}|^2,
\]
thus, by Gronwall's inequality, \eqref{MF_G1} is proved.
Multiplying the equation for $c^i_{ik}$ by $c^i_{ik}$ and summing over $i \neq k$ we obtain
\[  \begin{split}
\frac12 \frac{\di}{\di t} \sum_{i\neq k} |c^i_{ik}|^2 &=  \sum_{i\neq k} f^i_{ik} c^i_{ik} - \sum_{i\neq k} B( c)_{kk} |c^i_{ik}|^2 - \sum_{i,j \neq k} c^i_{ik} c^i_{ji}  c^j_{jk} - \sum_{\substack{i \neq k \\ j\neq i}} c^i_{ik} c^i_{jk}  c^j_{ji} \\
&\leq \sum_{i\neq k} f^i_{ik} c^i_{ik} - \sum_{i\neq k} B( c)_{kk} |c^i_{ik}|^2 - \tr(\hat B(c)^\trn B(c) \hat B(c)) - \sum_{\substack{i \neq k \\ j\neq i}} c^i_{ik} c^i_{jk}  c^j_{ji},
\end{split} \]
where $\hat B(c)_{hk} = B(c)_{hk}$ if $h = k$ and it is null otherwise. By \eqref{MF_G1}
\[
\bigg| \sum_{\substack{i \neq k \\ j\neq i}} c^i_{ik} c^i_{jk} c^j_{ji} \bigg| \leq \sup_\ell \sum_{i\neq \ell} |c^i_{i\ell}|^2 \bigg( \sum_{\substack{i \neq k \\ j\neq i}} |c^i_{jk}|^2 \bigg)^{\frac12} \leq \sqrt{\kappa_0}\, \sup_\ell \sum_{i\neq \ell} |c^i_{i\ell}|^2.
\]
We have
\[
\frac{\di}{\di t} \sum_{i\neq k} |c^i_{ik}|^2  \leq \kappa_f N^{-1} + (1+4M_+) \sum_{i\neq k} |c^i_{ik}|^2 + 2\sqrt{\kappa_0}\, \sup_\ell \sum_{i\neq \ell} |c^i_{i\ell}|^2;
\]
that is,
\[
\sup_k \sum_{i\neq k} |c^i_{ik}(t)|^2 \leq (\kappa_g + \kappa_f t) N^{-1} + \int_0^t \!(1+4M_++2\sqrt{\kappa_0}\,) \sup_k \sum_{i\neq k} |c^i_{ik}|^2.
\]
By Gronwall's inequality one finds \eqref{MF_estcallS}.
Consider at this point that $c^i_{ii}$ solves
\begin{equation} \label{MF_eqcii}
\dot c^i_{ii} + |c^i_{ii}|^2 + 2\sum_{j\neq i} c^j_{ji} c^i_{ij} = f^i_{ii},
\end{equation}
where by \eqref{MF_G1} and \eqref{MF_estcallS}
\[
\Big| \sum_{j\neq i} c^j_{ji} c^i_{ij} \Big|^2 \leq {\kappa_0\kappa_1} N^{-2}.
\]
Therefore, multiplying equation \eqref{MF_eqcii} by $c^i_{ii}$ one obtains
\[
\frac{\di}{\di t}\, |c^i_{ii}|^2 \leq \kappa_f + (2+M_+)|c^i_{ii}|^2 + {\kappa_0\kappa_1} N^{-2}
\]
and thus \eqref{MF_estcallC0} by Gronwall's inquality.
\end{proof}

\begin{lem} \label{MF_l2}
Under the hypotheses of \Cref{MF_thm}, let $c$ be as in \Cref{MF_l1}. Suppose that
\[
Me^{-2MT} > K_g, \qquad K_f < \frac{2M(Me^{-MT}-K_g)}{1-e^{-2MT}}.
\]
Then there exists $N_*(T)$ such that $B(c)(T) > -MI$ provided that $N > N_*(T)$. Furthermore, the map $T \mapsto N_*(T)$ is continuous on $[0,+\infty)$.
\end{lem}

\begin{proof}
Consider that $B(c)$ solves the equation
\begin{equation} \label{MF_eqB}
\dot B(c) + B(c)^2 = B(f) - D. 
\qquad \text{where} \quad 
D_{ik} = \sum_{j\neq i} c^i_{kj} c^j_{ji}.
\end{equation}
By estimates \eqref{MF_G1} and \eqref{MF_estcallS},
\[
\norm{D}_2^2 = \sum_{i,k} \bigg( \sum_{j\neq i} c^i_{kj} c^j_{ji} \bigg)^2 \leq \sup_i \sum_{j \neq i} |c^j_{ji}|^2 \cdot \sum_{\substack{i,j,k \\ j \neq i}} |c^i_{jk}|^2 \leq \kappa_0 \kappa_1 N^{-1}.
\]
Let now $\xi$ solve the linear equation $\dot \xi = B(c)^\trn \xi$ on $[0,T)$ with terminal condition $\xi(T) = \zeta$, for some arbitrary $\zeta \in \bb S^{N-1}$; note that since $\frac\di{\di t} |\xi|^2 \geq -4M|\xi|^2$ we have
\[
1 \wedge e^{2M(T-\cdot)} \leq |\xi| \leq 1 \vee e^{2M(T-\cdot)} \quad \text{on $[0,T]$}.
\]
By \eqref{MF_eqB} we get
\[
(\xi^\trn B(c) \xi)\dot{\vphantom{|}}\, = \xi^\trn (B(f)-D + B(c)B(c)^\trn) \xi \geq \xi^\trn (B(f)-D) \xi \geq -(K_f + \sqrt{\kappa_0\kappa_1} N^{-\frac12})|\xi|^2.
\]
If $K_f < 0$ and $N$ is large enough, then $B(c)(T) > - \nu K_g I$, where
\[
\nu  \defeq \begin{cases}
e^{2MT} & \text{if $MK_g \geq 0$} \\
1 & \text{if $M K_g \leq 0$},
\end{cases}
\]
thus it is easily seen that $B(c)(T) > - MI$.
If $K_f \geq 0$, then
\begin{equation} \label{MF_estBc}
B(c)(T) \geq - \Big( \nu K_g + \int_0^T ( K_f + \sqrt{\kappa_0\kappa_1} N^{-\frac12} ) \, e^{2M(T-\cdot)} \Big) I .
\end{equation}
It follows that in order to have $B(c)(T) > -MI$ it suffices that
\[
\nu K_g + T K_f h(MT) + N^{-\frac12} T \sqrt{\kappa_0(T)\kappa_1(T)}\,h(MT) < M,
\]
where $h(z) \defeq (e^{2z}-1)/(2z)$. This is guaranteed by our assumptions on $K_g$, $K_f$ and $M$ (by which $\nu K_g + T K_f h(MT) < M$), provided that $N$ is large enough. The continuity of $N_*(T)$ is easily seen as one can write it explicitly using the estimates above.
\end{proof}

\begin{proof}[Proof of \Cref{MF_thm}]
Fix $M$ according to \Cref{MF_l2} (note that this implies that $M > K_g$). By the Cauchy--Lipschitz theorem there exists $\tau > 0$ such that \eqref{MF_sys} has a unique absolutely continuous solution on $[0,\tau)$.  By taking $\tau$ smaller if necessary, we may suppose that by continuity $B(c) > -M I$ on $[0,\tau)$. Then
\[
\bar \tau \defeq \sup\{ \tau > 0 :\ \text{\eqref{MF_sys} has a unique absolutely continuous solution with $B(c) > - MI$ on $[0,\tau)$} \}
\]
is well-defined. Seeking for a contradiction, suppose that $\bar \tau < T$. Let $N_0 \defeq \max_{[0,T]} N_*$, where $N_*$ is given by \Cref{MF_l2}.
By \Cref{MF_l1}, $c$ continuously extends on $[0,\bar\tau]$, with $B(c)(\bar\tau) > -MI$ as guaranteed by \Cref{MF_l2}, thanks to our choice of $N_0$. By the Cauchy--Lipschitz theorem one can extend the solution on $[0,\tau')$ for some $\tau' > \bar\tau$ and by continuity we may suppose that $B(c) > -MI$ on $[0,\tau')$. This contradicts the maximality of $\bar\tau$, thus $\bar\tau \geq T$. Finally, estimate \eqref{MF_estCN} follows from \eqref{MF_G1}, \eqref{MF_estcallS} and \eqref{MF_estcallC0}. This concludes the proof.
\end{proof}



\end{document}